\documentclass[a4paper,11pt]{article}
\usepackage[latin1]{inputenc}
\usepackage[english]{babel}
\usepackage{amsmath,amsthm}
\usepackage{amssymb}
\usepackage{ifpdf}
\ifpdf
  \usepackage[pdftex]{graphicx}
  \usepackage{epstopdf}
\else
  \usepackage[dvips]{graphicx}
\fi
\usepackage[pdftitle={The PDF title},
  pdfauthor={Your Name},
  pdffitwindow=true,
  breaklinks=true,
  colorlinks=true,
  urlcolor=blue,
  linkcolor=red,
  citecolor=red,  
  anchorcolor=red]{hyperref}
\usepackage{cases}
\usepackage{booktabs}
\usepackage{subfig}
\usepackage{algorithm}

\newcommand{\f}{\frac}
\newcommand{\p}{\partial}
\newcommand{\bs}{\boldsymbol}

\newtheorem{theorem}{Theorem}

\newtheorem{lemma}{Lemma}
\newtheorem{definition}{Definition}

\newcommand{\wt}{\widetilde}
\newcommand\ubar[1]{\underline{#1}}

\title{Fourth order finite difference methods for the wave equation with mesh refinement interfaces}

\begin{document}
\author{Siyang Wang \thanks{Department of Mathematical Sciences, Chalmers University of Technology
    and University of Gothenburg, SE-412 96 Gothenburg,
    Sweden. \href{mailto:siyang.wang@chalmers.se}{Email: siyang.wang@chalmers.se}}\and N. Anders
  Petersson \thanks{Center for Applied Scientific Computing, Lawrence Livermore National Laboratory,
    Livermore, CA 94551, USA. \href{mailto:petersson1@llnl.gov}{Email: petersson1@llnl.gov}}}
\maketitle

\begin{abstract}
%We analyze two types of summation-by-parts finite difference operators for solving the
%two-dimensional wave equation on a grid with a mesh refinement interface. The first type uses ghost
We analyze two types of summation-by-parts finite difference operators for approximating the
second derivative with variable coefficient. The first type uses ghost
points, while the second type does not use any ghost points. A previously unexplored relation
between the two types of summation-by-parts operators is investigated. By combining them we develop
a new fourth order accurate finite difference discretization with hanging nodes on the mesh refinement interface. 
We take the model problem as the two-dimensional acoustic wave equation in second order form in terms of acoustic pressure,  
and prove energy stability for the proposed method. Compared to previous approaches using ghost points, the
proposed method leads to a smaller system of linear equations that needs to be solved for the ghost
point values. Another attractive feature of the proposed method is that the explicit time step does not
need to be reduced relative to the corresponding periodic problem. Numerical experiments, both for
smoothly varying and discontinuous material properties, demonstrate that the proposed method
converges to fourth order accuracy. A detailed comparison of the accuracy and the time-step
restriction with the simultaneous-approximation-term penalty method is also presented.
\end{abstract}

\section{Introduction}
Based on the pioneering work by Kreiss and Oliger~\cite{Kreiss1972}, it is by now well known that
high order accurate ($\geq 4$) numerical methods for solving hyperbolic partial differential
equations (PDE) are more efficient than low order methods. While Taylor series expansion can easily
be used to construct high order finite difference stencils for the interior of the computational
domain, it is in general difficult to find stable boundary closures that avoid spurious growth in time of the
numerical solution. Finite difference operators that satisfy the summation-by-parts (SBP) identity,
first introduced by Kreiss and Scherer~\cite{Kreiss1974}, provide a recipe for achieving both
stability and high order accuracy.

%In this paper we use finite difference
%operators that satisfy the summation-by-parts (SBP) property, first introduced by Kreiss and
%Scherer~\cite{Kreiss1974}, to solve the two-dimensional wave equation with variable coefficients on a grid with
%anon-conforming mesh refinement interface.

An SBP operator is constructed such that the energy estimate of the continuous PDE can be carried
out discretely for the finite difference approximation, with summation-by-parts replacing the
integration-by-parts principle. As a consequence, a discrete energy estimate can be obtained to
ensure that the discretization is energy stable. When deriving a continuous energy estimate, the
boundary terms resulting from the integration-by-parts formula are easily controlled through the
boundary conditions. The fundamental benefit of using SBP operators is that a discrete energy
estimate can be derived in a similar way. Here, the summation-by-parts identities result in
discrete boundary terms. These terms dictate how the boundary conditions must be discretized to guarantee
energy stability for the finite difference approximation.

We consider the SBP discretization of the two-dimensional acoustic wave equation on Cartesian grids, and focus on the case when the material properties are discontinuous in a semi-infinite domain. To obtain high order accuracy, one  approach is to decompose the domain into multiple subdomains, such that the material is smooth within each
subdomain. The governing equation is then discretized by SBP operators in each subdomain, and
patched together by imposing interface conditions at the material discontinuity. For computational
efficiency, the mesh size in each subdomain should be chosen inversely proportional to the wave 
speed \cite{Hagstrom2012,Kreiss1974}, leading to mesh refinement interfaces with hanging nodes. 

We develop two approaches for imposing interface conditions in the SBP finite difference framework.
In the first approach, interface conditions are imposed strongly by using ghost points. In this
case, the SBP operators also utilize ghost points in the difference approximation. We call this the
SBP-GP method.  In the second approach, the SBP-SAT method, interface conditions are imposed weakly
by adding penalty terms, also known as simultaneous-approximation-terms
(SAT)~\cite{Carpenter1994}. The addition of penalty terms in the SBP-SAT method bears similarities
with the discontinuous Galerkin method \cite{Hesthaven2008}. A high order accurate SBP-SAT
discretization of the acoustic wave equation in second order form was previously developed by Wang
et al.~\cite{Wang2016}. Petersson and Sj\"ogreen~\cite{Petersson2010} developed a second order
accurate SBP-GP scheme for the elastic wave equation in displacement formulation with mesh
refinement interfaces. We note that the projection method \cite{Olsson1995a, Olsson1995b} could in
principle also be used to impose interface conditions, but will not be considered here.

%In the SBP finite difference framework, boundary and interface 
%conditions can be imposed either strongly or weakly. For strong imposition, ghost points outside the domain
%\cite{Sjogreen2012}. In this case, the SBP operators also utilize the ghost points for difference
%approximations. We call this the SBP-GP method. In the second approach, called SBP-SAT, boundary
%conditions are imposed weakly by adding penalty terms, also known as
%simultaneous-approximation-terms (SAT) \cite{Carpenter1994}, to the discretization. Thus, the
%SBP-SAT method bears similarities with the discontinuous Galerkin method
%\cite{Appelo2015,Grote2006}.  For the wave equation with non-conforming mesh refinement interfaces,
%a high order accurate SBP-SAT finite difference method and a second order accurate SBP-GP method
%were previously developed in \cite{Wang2016} and \cite{Petersson2010}, respectively.

In this paper, we present two ways of generalizing the SBP-GP method in \cite{Petersson2010} to
fourth order accuracy. The first approach is a direct generalization of the second order accurate
technique. It imposes the interface conditions using ghost points from both sides of the mesh
refinement interface.  The second approach is based on a previously unexplored relation between SBP
operators with and without ghost points. This relation allows for an improved version of the fourth
order SBP-GP method, where only ghost points from one side of the interface are used to impose the
interface conditions. This approach reduces the computational cost of updating the solution at the
ghost points and should also simplify the generalization to three-dimensional problems.

Even though both the SBP-GP and SBP-SAT methods have been used to solve many kinds of PDEs, the
relation between them has previously not been explored. An additional contribution of this paper is to connect
the two approaches, provide insights into their similarities and differences, as well as making a comparison in
terms of their efficiency. 

The remainder of the paper is organized as follows. In Section~\ref{sec_sbp}, we introduce the SBP
methodology and present the close relation between the SBP operators with and without ghost
points. In Section~\ref{sec_bc}, we derive a discrete energy estimate for the wave equation in one
space dimension with Dirichlet or Neumann boundary conditions. Both the SBP-GP and the SBP-SAT
methods are analyzed in detail and their connections are discussed. In Section~\ref{sec_mr}, we
consider the wave equation in two space dimensions, and focus on the numerical treatment of grid
refinement interfaces with the SBP-GP and SBP-SAT methods. Numerical experiments are conducted in
Section~\ref{sec_num}, where we compare the SBP-GP and SBP-SAT methods in terms of their time-step
stability condition and solution accuracy. Our findings are summarized in Section~\ref{sec_conc}.

% no tilde indicates it is SBP-SAT without ghost points
% tilde indicates it is SBP-GP with ghost points
\section{SBP operators}\label{sec_sbp}
%We begin with preliminaries that will be used in the discussion of SBP finite difference methods.
Consider the bounded one-dimensional domain $x\in[0,1]=:\Omega$ and the uniform grid on $\Omega$,
\[
\bs{x}=[x_1,\cdots,x_n]^T,\quad x_j = (j-1) h,\quad j = 1,2,\cdots,n,\quad h=1/(n-1).
\]
The grid points in $\bs{x}$ are either in the interior of $\Omega$, or on its boundary.  We also define
two ghost points outside of $\Omega$: $x_{0}=-h$ and $x_{n+1}=1+h$. Let the vector
$\wt{\bs{x}}=[x_0,\cdots,x_{n+1}]^T$ denote the grid with ghost points. Throughout this paper, we will
use the tilde symbol to indicate that ghost points are involved in a grid, a grid function, or in a difference operator. 

We consider a smooth function $u(x)$ in the domain $\Omega$, and define the grid function $u_j:=u(x_j)$. Let
\begin{equation}\label{grid_uv}
\bs u=[u_1,\cdots,u_n]^T \text{ and } \bs v=[v_1,\cdots,v_n]^T
\end{equation}
denote real-valued grid functions on $\bs x$, and let
\begin{equation}\label{grid_tilde_uv}
\wt{\bs u}=[u_0, \bs{u}^T, u_{n+1}]^T \text{ and } \wt{\bs v}=[v_0, \bs{v}^T, v_{n+1}]^T
\end{equation}
denote the corresponding real-valued grid functions on $\wt{\bs x}$.
%In the context of SBP identities, the values of
%the grid functions are arbitrary. However, when analyzing truncation errors, we assume the grid
%functions are sufficiently smooth functions evaluated on the grid.

%In the discussion of SBP operators, we will frequently use the grid without ghost points, and the corresponding grid functions. 
%We use a tilde sign, for example $\widetilde{\bs{x}}=[x_1,\cdots,x_n]^T$, to indicate that ghost points are excluded. Similarly, the
%grid functions on $\wt{\bs{x}}$ are defined as
%\begin{equation}\label{grid_tilde_uv}
%\wt{\bs u}=[u_1,\cdots,u_n] \text{ and }\wt{\bs v}=[v_1,\cdots,v_n],
%\end{equation}
%by removing values on the ghost points in $\bs u$ and $\bs v$, respectively. 

We denote the standard discrete $L^2$ inner product by
\begin{equation*}
({\bs u},{\bs v})_2 = h\sum_{j=1}^n u_j v_j.
\end{equation*}
For SBP operators, we need the weighted inner product 
\begin{equation}\label{wip}
({\bs u},{\bs v})_{h} = h\sum_{j=1}^n w_j u_j v_j,\quad w_j \geq \delta > 0,
\end{equation}
where $\delta$ is a constant, $w_j=1$ in the interior of the domain and $w_j\neq 1$ at a few grid
points near each boundary.  The number of grid points with $w_j\neq 1$ is independent of $n$, but
depends on the order of accuracy of the SBP operator. Let $\|\cdot \|_h$ be the SBP norm induced
from the inner product $(\cdot,\cdot)_{h}$.  Furthermore, let the diagonal matrix $W$ have entries
$W_{jj}=hw_j>0$. Then, in matrix-vector notation, $(\bs{u},\bs{v})_h = \bs{u}^T W \bs{v}$.

The SBP methodology was introduced by Kreiss and Scherer in \cite{Kreiss1974}, where the first
derivative SBP operator $D\approx \p/\p x$ was also constructed. The operator $D$ does not use ghost points,
and satisfies the first derivative SBP identity.
\begin{definition}[First derivative SBP identity]
The difference operator $D$ is a first derivative SBP operator if it satisfies
\begin{equation}
\label{sbp1}
({\bs u}, D{\bs v})_{h} = -(D{\bs u},{\bs v})_{h}-u_1 v_1 + u_nv_n,
\end{equation}
for all grid functions $\bs{u}$ and $\bs{v}$.
\end{definition}
We note that \eqref{sbp1} is a discrete analogue of the integration-by-parts formula
\[
\int_0^1 u\f{dv}{dx} dx = -\int_0^1 \f{du}{dx}v-u(0)v(0)+u(1)v(1).
\]

Centered finite difference stencils are used on the grid points away from the boundaries, where the weights in the SBP norm are equal to one. To retain the SBP identity, special one-sided
boundary stencils must be employed at a few grid points near each boundary. Kreiss and Scherer
showed in \cite{Kreiss1974} that the order of accuracy of the boundary stencil must be lower than in
the interior stencil. With a diagonal norm and a $2p^{th}$ order accurate interior stencil, the
boundary stencil can be at most $p^{th}$ order accurate. The overall convergence rate can be between $p+1/2$ and
$2p$, depending on the equation and the numerical treatment of boundary and interface conditions \cite{Gustafsson1975,Wang2017a,Wang2017b}. 
In the following we refer to the accuracy of an SBP operator by its interior order of accuracy ($2p$). 

It is possible to construct block norm SBP operators with $2p^{th}$ order interior stencils and
$(2p-1)^{th}$ order boundary stencils. Despite their superior accuracy, the block norm SBP operators
are seldomly used in practice because of stability issues related to variable coefficients. However,
in some cases the block norm SBP operators can be stabilized using artificial dissipation~\cite{Mattsson2013}.

For second derivative SBP operators, we focus our discussion on discretizing the expression
\begin{equation}\label{eq_2nd-der-var-coeff}
\f{d}{d x} \left(\mu(x)\f{d v}{d x}(x)\right).
\end{equation}
Here, the smooth function $\mu(x)>0$ may represent a variable material property or a metric
coefficient. In the following we introduce two different types of second derivative
SBP operators that are based on a diagonal norm. The first type uses one ghost point outside each boundary, while the second type does
not use any ghost points. We proceed by explaining the close relation between these operators. To make
the presentation concise, we exemplify the relation for the case of fourth order accuracy ($2p=4)$.

\subsection{Second derivative SBP operators with ghost points}
Sj\"ogreen and Petersson~\cite{Sjogreen2012} derived a fourth order accurate SBP discretization $\wt
G(\mu)\wt{\bs v}$ for approximating \eqref{eq_2nd-der-var-coeff}.
%\approx \f{\p}{\p x} (\mu(x) \f{\p}{\p x}u(x))$
This discretization was originally developed for solving the seismic wave equations and is
extensively used in the software package SW4~\cite{Sw4}. The formula is based on a five-point
centered difference stencil of fourth order accuracy in the interior of the domain. Special
one-sided boundary stencils of second order accuracy are used at the first six grid points
near each boundary. Note, in particular, that $\wt{G}(\mu)\wt{\bs{v}}$ uses the ghost point values of
$\wt{\bs{u}}$ to approximate \eqref{eq_2nd-der-var-coeff} on the boundary itself, as 
illustrated in Figure~\ref{G_GhostPoint}.
%, where the structure of $\wt G(\mu)$ is shown when the operator
%is represented by a matrix of size $30\times 32$ on a grid with 30 grid points.
\begin{figure}
\centering
\includegraphics[width=.48\linewidth]{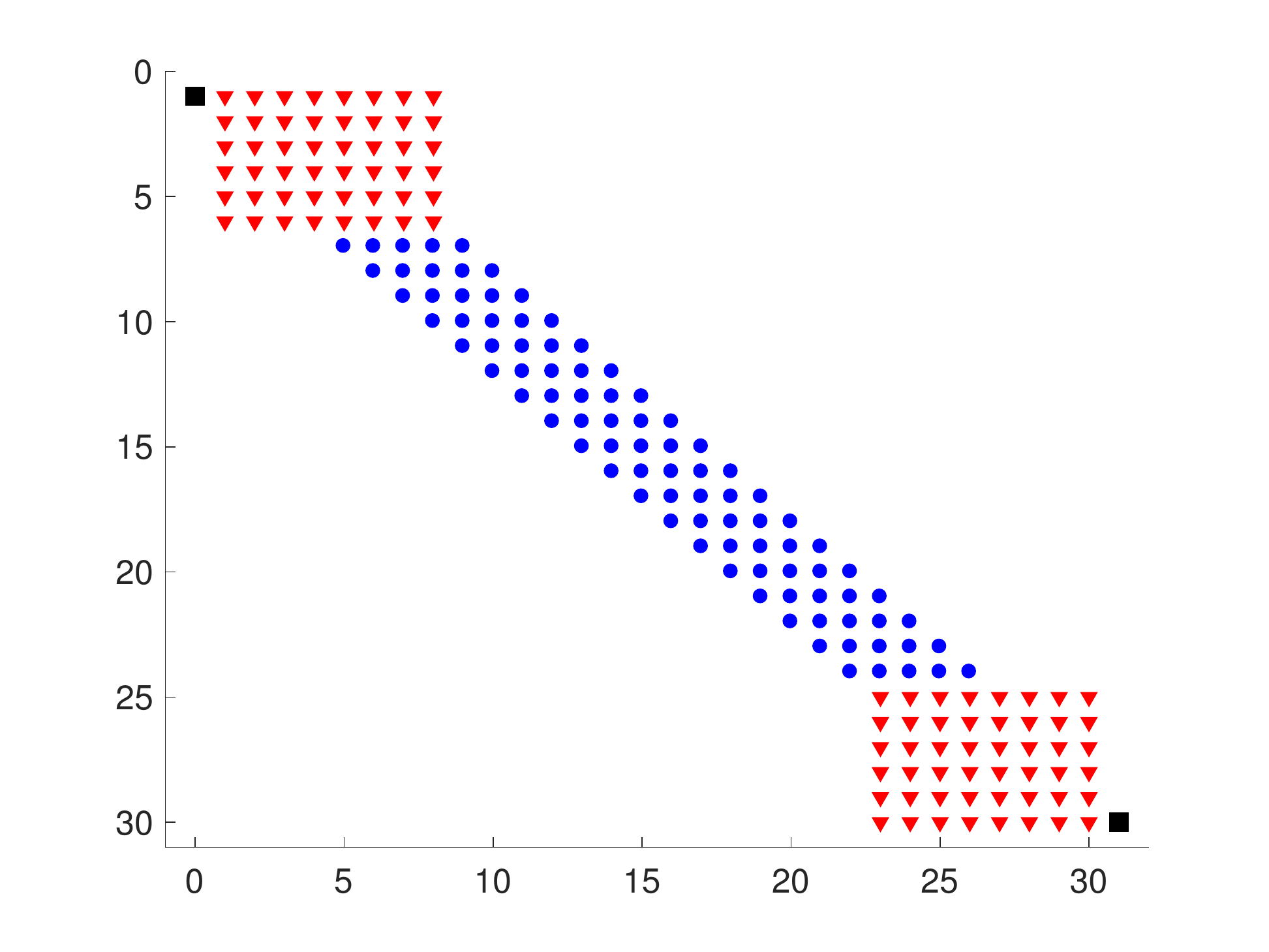}
%\subfloat[][]{\includegraphics[width=.48\linewidth]{G_GhostPoint.eps}\label{G_GhostPoint}}
%\subfloat[][]{\includegraphics[width=.48\linewidth]{G_NoGhostPoint.eps}\label{G_NoGhostPoint}}
\caption{The non-zero coefficients of the SBP operator $\wt G(\mu)$ in matrix form, for a grid with
  30 grid points. Blue circles: standard five-point difference stencil. Red triangles: special
  boundary stencil. Black squares: ghost points. The structure of $G(\mu)$ is the same, but without
  the black squares. Note that the grid function $\wt{G}(\mu) \wt{\bs v}$ is defined at the same
  grid points as $\bs{v}$. }
  \label{G_GhostPoint}
\end{figure}
As will be shown below, the difference approximation of the wave equation is energy stable because
the difference operator $\wt{G}(\mu)$ satisfies the second derivative SBP identity.
\begin{definition}[Second derivative SBP identity]\label{def_SBP2}
The difference operator $\wt{G}(\mu)$ is a second derivative SBP operator if it satisfies
\begin{equation}\label{SBP_GP}
(\bs u, \wt G(\mu) \wt{\bs v})_{h} = -S_{\mu}(\bs u,\bs v) - u_1 \mu_1 \wt{\bs{b}}_1^T \wt{\bs{v}} + u_n \mu_n \wt{\bs{b}}_n^T \wt{\bs{v}},
\end{equation}
for all grid functions $\bs{u}$ and $\wt{\bs v}$. Here, $\mu_1=\mu(x_1)$, $\mu_n=\mu(x_n)$ and the bilinear form $S_{\mu}(\cdot,\cdot)$
is symmetric and positive semi-definite. The boundary difference formulas $\wt{\bs{b}}_1^T\wt{\bs v}$ and
$\wt{\bs{b}}_n^T\wt{\bs v}$ approximate $dv/dx$ at $x_1$ and $x_n$, making use of
the ghost point values $v_0$ and $v_{n+1}$, respectively.
\end{definition}
We remark that the boundary difference operators, $\wt{\bs{b}}_1^T$ and $\wt{\bs{b}}_n^T$, are
constructed with fourth order accuracy in~\cite{Sjogreen2012}. Note that \eqref{SBP_GP} is a discrete analogue of the integration-by-parts formula
\[
\int_0^1 u \f{d}{dx} \left(\mu \f{dv}{dx}\right) dx = -\int_0^1 \mu\f{du}{dx}\f{dv}{dx}-u(0)\mu(0)\f{dv}{dx}(0)+u(1)\mu(1)\f{dv}{dx}(1).
\]

%The boundary derivative
%\begin{equation*}
%\bs{b}_1^T \bs{v} = \f{1}{h} \sum_{j=0}^4 \sigma_j v_j 
%\end{equation*}
%is a fourth order accurate approximation of $V_x(x_1)$, and makes use of the ghost point value
%$v_0$. Similarly, $\bs{b_n^T v}= V_x(x_n)+\mathcal{O}(h^4)$ uses the ghost point value $v_{n+1}$. 
%We emphasize that the bilinear form $S_{\mu}$ does not depend on any ghost point values.  The SBP
%operator $G(\mu)$ only uses the ghost point to approximate the second derivative on the boundaries
%$x_1$ and $x_n$. 

%Prior to SW4, a second order accurate
%ghost point technique was developed in \cite{Nilsson2007} and implemented in the WPP
%code~\cite{WPP}.

%The second and fourth order accurate operators of this type are constructed in \cite{Nilsson2007}
%and \cite{Sjogreen2012}, respectively. The operator $B_0$ is constructed to have the same accuracy
%order as the interior stencil, i.e. $2p^{th}$ order accurate.

%When $G(\mu)$ is used to solve a time dependent PDE, the numerical solution is evolved by the
%equation on all but the ghost point. The ghost point acts as an additional degree of freedom to
%impose the boundary condition, and its value is determined by the boundary condition and the numerical scheme.

\subsection{Second derivative SBP operators without ghost points}
The second type of second derivative SBP operator, denoted by $G_{2p}(\mu)$, does not use any
ghost points. This type of operator was constructed by Mattsson~\cite{Mattsson2012} for the cases
of second, fourth and sixth order accuracy ($2p = 2,4,6$). In the following discussion we focus on the fourth order
case and define $ G(\mu) =  G_{4}(\mu)$.

In the interior of the domain, the operator $ G(\mu)$ uses the same five-point wide,
fourth order accurate stencil as the operator with ghost points, $\wt G(\mu)$. At the first six grid
points near the boundaries, the two operators are similar in that they both use a second order
accurate one-sided difference stencil that satisfies an SBP identity of the form \eqref{SBP_GP}, but
without ghost points,
\begin{equation}\label{SBP_NGP}
({\bs u}, G(\mu) {\bs v})_{h} = - S_{\mu}({\bs u}, {\bs v})_{h} - u_1 \mu_1
  {\bs{b}}_1^T {\bs{v}} + u_n \mu_n {\bs{b}}_n^T {\bs{v}}.
\end{equation} 
Similar to \eqref{SBP_GP}, the bilinear form $S_{\mu}(\cdot,\cdot)$ is symmetric and positive
semi-definite. In this case, the boundary difference operators ${\bs b}_1^T$ and ${\bs b}_n^T$ are constructed
with third order accuracy, using stencils that do not use any ghost points. The structure of $G(\mu)$
is the same as shown in Figure~\ref{G_GhostPoint}, but without the two black squares representing the
ghost points.

%B_0$ approximates the first derivative without using any ghost point. As a consequence, there is no
%degree of freedom for the boundary condition, and it is more natural to impose by adding penalty
%terms. The $2p^{th},\ p=1,2,3$ order SBP operators of this type are constructed in
%\cite{Mattsson2012}, where the order of accuracy of $\widetilde B_0$ is $p+1$.

\subsection{The relation between SBP operators with and without ghost points}
When using the SBP operator $\wt G(\mu)$ with ghost points, boundary conditions are imposed in a
strong sense by using the ghost point values as additional degrees of freedom. On the other hand,
for the SBP operator $G(\mu)$ without ghost points, boundary conditions are imposed weakly by using
a penalty technique. Though these two types of SBP operators are used in different ways, they are
closely related to each other. In fact, an SBP operator with ghost points can easily be modified
into a new SBP operator that does not use any ghost points, and vice versa. The new operators
preserve the SBP identity and the order of accuracy of the original operators. In the following, we
demonstrate this procedure for the fourth order accurate version of $\wt G(\mu)$~\cite{Sjogreen2012}
and $G(\mu)$ \cite{Mattsson2012}. For simplicity, we only consider the stencils near the left
boundary. The stencils near the right boundary can be treated in a similar way.

To discuss accuracy, let us assume that the grid function $\wt{\bs v}$ is a restriction of a sufficiently smooth 
function $V(x)$ on the grid $\wt{\bs x}$. 
The boundary difference operator associated with $\wt{G}(\mu)$ satisfies
%\begin{equation}\label{B1}
%\bs{b_1^T v} = \frac{1}{12h}(-3v_{0}-10v_1+18v_2-6v_3+v_4)=V_x(x_1)+\f{93h^4}{240}\f{d^5}{dx^5}V(x_1)+\mathcal{O}(h^5).
%\end{equation}
\begin{equation}\label{B1}
\wt{\bs b}_1^T \wt{\bs v} = \frac{1}{12h}(-3v_{0}-10v_1+18v_2-6v_3+v_4)=\f{dV}{dx}(x_1)+\mathcal{O}(h^4).
\end{equation}
Let's consider the modified boundary difference operator, 
\begin{equation}\label{B1n}
%  \wt{\ubar{\bs{b}}}_1^T \wt{\bs v} =
  \wt{\bs b}_1^T \wt{\bs v} + \beta h^4 \wt{\bs d}^T_{5+ }\wt{\bs v},
\end{equation}
where
\begin{equation}\label{D5}
\wt{\bs d}^T_{5+ }\wt{\bs{v}}= \frac{1}{h^5} (-v_0+5v_1-10v_2+10v_3-5v_4+v_5)=\f{d^5 V}{dx^5}(x_1)+\mathcal{O}(h)
\end{equation}
is a first order accurate approximation of the fifth derivative at the boundary point $x_1$.
%By Taylor expansion, $\bs{d^T_{5+ }v}$ is also a first order accurate approximation of $\f{d^5
%  V}{dx^5}(x_1)$.
Both the approximations \eqref{B1} and \eqref{B1n} are exact at $x_1$ if
$V(x)$ is a polynomial of order at most four. For any (finite) value of $\beta$, \eqref{B1n} is a
fourth order accurate approximation of $\f{dV}{dx}(x_1)$.

We note that the coefficient of $v_0$ in \eqref{B1} is -1/4. To eliminate the dependence on
$v_{0}$ in \eqref{B1n}, we choose $\beta = -1/4$ and define a new boundary difference
operator by
\[
\wt{\ubar{\bs{b}}}_1^T \bs{v}  =
\frac{1}{12h}(-25v_1+48v_2-36v_3+16v_4-3v_5)=V_x(x_1)+\mathcal{O}(h^4).
\]
This stencil does not use the ghost point value $v_0$. Instead, it uses the value
$v_5$, which is not used by $\wt{\bs{b}}_1^T\wt{\bs{v}}$. Here and throughout the paper, we use an underbar to
indicate operators that have been modified by adding/removing ghost points.

To retain the SBP identity~\eqref{SBP_GP}, the operator $\wt G(\mu)$ must be changed accordingly. We
can maintain the same bilinear form $S_{\mu}(\cdot,\cdot)$ if we only modify $\wt{G}(\mu)$ on the
boundary itself. We make the ansatz
\begin{equation}\label{eq_g-tilde-mod}
\ubar{\wt{G}}_1(\mu) \bs{v} = \wt{G}_1(\mu) \wt{\bs{v}} + \wt{\bs{a}}^T \wt{\bs{v}},
\end{equation}
where $\wt{G}_1(\mu) \wt{\bs{v}}$ should be interpreted as the first element of vector $\wt{G}(\mu) \wt{\bs{v}}$. 
To see the relation between $\ubar{\wt{G}}_1(\mu) \bs{v}$ and $\wt{G}_1(\mu) \wt{\bs{v}}$ in the SBP identity  \eqref{SBP_GP}, we pick a particular grid function $\bs{u}$ in \eqref{SBP_GP} satisfying $u_1=1$ and $u_j=0$, for
$j\geq 2$. The balance between the left and right hand sides of that equation is maintained if
\[
h w_1 \wt{\bs{a}}^T\wt{\bs{v}} = - \beta h^4 \mu_1\wt{\bs{d}}^T_{5+ }\wt{\bs{v}}\quad \Rightarrow\quad
\wt{\bs{a}}^T\wt{\bs{v}} = \f{12}{17}h^3 \mu_1\wt{\bs{d}}^T_{5+ }\wt{\bs{v}}.
\]
Here we have used that $\beta=-1/4$ and that $w_1 = 17/48$ is the weight of the SBP norm at the first
grid point. The ghost point value $v_0$ is only used by $\wt{G}(\mu)\wt{\bs v}$ on the boundary
itself. It satisfies
\begin{equation}\label{eq_sbp-gp-1}
\wt{G}_1(\mu)\wt{\bs v} = \frac{1}{h^2} \sum_{k=1}^8 \sum_{m=1}^8 \beta_{k,m}\mu_m v_k +
\frac{12}{17}\f{\mu_1}{h^2} v_0,
\end{equation}
where $\beta_{k,m}$ are constants~\cite{Sjogreen2012} (the numerical values can be found in the open
source code of SW4~\cite{Sw4}). Because the coefficient of $v_0$ in $\wt{\bs{d}}^T_{5+}\wt{\bs{v}}$ is
$-1/h^5$, the dependence on $v_0$ cancels in \eqref{eq_g-tilde-mod}. This cancellation is a consequence of  the operators $\wt{G}(\mu)$ using ghost points only from 
$\wt{\bs{b}}_1^T$ but not $S_{\mu}(\cdot,\cdot)$, see~\cite{Sjogreen2012} for details.

The new SBP difference operator that does not use ghost points can be written as
\begin{align*}
\ubar{\wt{G}}_1(\mu) \bs{v} &= \frac{1}{h^2} \sum_{k=1}^8 \sum_{m=1}^8 \beta_{k,m}\mu_m v_k +
\f{12}{17}\f{\mu_1}{h^2} \left(5v_1 - 10 v_2 + 10 v_3 - 5 v_4 + v_5\right),\\
\ubar{\wt{G}}_j(\mu) \bs{v} &= \wt{G}_j(\mu) \wt{\bs{v}},\quad j=2,3,\ldots.
\end{align*}
Note that the second equation is satisfied independently of the ghost point value, $v_0$.

To emphasize that $\wt{\ubar{G}}(\mu)$ is modified from $\wt{G}(\mu)$, we keep the tilde symbol on
$\wt{\ubar{G}}(\mu)$, even though the operator does not use any ghost points. The new operator pair
$(\ubar{\wt{G}}(\mu), \ubar{\wt{\bs{b}}}_1)$ shares important properties with the original operator
pair $(\wt{G}(\mu), {\wt{\bs{b}}}_1)$. In particular, both pairs satisfy the SBP
identity~\ref{def_SBP2} and have the same orders of accuracy in the interior and near each
boundary. Even though the SBP operator $\wt{\ubar{G}}(\mu) $ does not use any ghost points, it is
not the same as the SBP operator $G(\mu)$ constructed by Mattssson~\cite{Mattsson2012}.
% AP reformulated this sentence
The dissimilarity arises because the corresponding boundary difference operators are constructed
with different orders of accuracy.

% AP: No transpose is needed on b_1
For the SBP operator pair $(G(\mu), \bs{b}_1)$ that does {\em not} use ghost points, we can
reverse the above procedure to derive a new pair of SBP operator that uses a ghost point. The
boundary difference operator associated with ${G}(\mu)$ is
\begin{equation}\label{s3}
\bs{b}_1^T \bs{v} = \frac{1}{6h}(-11v_1+18v_2-9v_3+2v_4) = \f{dV}{dx}(x_1)+\mathcal{O}(h^3).
\end{equation}
Another third order approximation of $dV/dx(x_1)$ is given by the difference formula
\begin{equation}\label{s3_n1}
%  \ubar{\bs{b}}_1^T \wt{\bs v} =
  {\bs{b}}_1^T {\bs{v}} + \gamma h^3 \wt{\bs d}_{4+}^T\wt{\bs v},
\end{equation}
where
\begin{equation}\label{d4}
\wt{\bs d}_{4+}^T\wt{\bs v} = \frac{1}{h^4} (v_0-4v_1+6v_2-4v_3+v_4) = \f{d^4 V}{dx^4}(x_1)+\mathcal{O}(h).
\end{equation}
The boundary operator \eqref{s3} is exact for any polynomial $V(x)$ of order at most three and
$\wt{\bs{d}}_{4+}^T\wt{\bs{v}}=0$ for such polynomials. Therefore, \eqref{s3_n1} is third order
accurate for any value of $\gamma$. By choosing $\gamma=-1/3$, we obtain a new boundary difference
operator that uses the ghost point value $v_0$, but does not depend on $v_4$,
\begin{equation}\label{eq_bndryder_wgp}
\ubar{\bs{b}}_1^T \wt{\bs v}  :=  \bs{b}_1^T \bs{v} - \frac{1}{3} h^3 \wt{\bs d}_{4+}^T\wt{\bs v} =
\frac{1}{6h}(-2v_0-3v_1+6v_2-v_3) = \f{dV}{dx}(x_1)+\mathcal{O}(h^3).
\end{equation} 
As a result, the new boundary difference operator has the minimum stencil width for a third order
accurate approximation of a first derivative.
% AP: this is stated at the end of this section:
%To emphasize that $\ubar{\bs{b}}^T_1$ is modified
%from $\bs{b}^T_1$, we have not added a tilde symbol on ${\ubar{\bs{b}}}^T_1$ even though it
%uses a ghost point.

To satisfy the SBP identity~\eqref{SBP_GP} for difference operators that include ghost points, we
must modify $G(\mu)$ to be compatible with the new boundary difference operator
$\ubar{\bs{b}}^T_1$. As before, we consider a grid function $\bs{u}$ with $u_1=1$ and $u_j=0$,
for $j\geq 2$. To maintain the balance between the left and right hand sides of \eqref{SBP_GP}, the
following must hold
\begin{equation}\label{eq_without_with_gp}
\ubar{G}_1(\mu) \wt{\bs{v}} := 
G_1(\mu) \bs{v}- \frac{\gamma h^3}{w_1 h} \mu_1\wt{\bs{d}}_{4+}^T\wt{\bs{v}}=
G_1(\mu) \bs{v}+\frac{16}{17} h^2\mu_1\wt{\bs{d}}_{4+}^T\wt{\bs{v}}.
\end{equation}
The new SBP operator that uses a ghost point becomes
\[
\ubar{G}_j(\mu) \wt{\bs{v}} =
\begin{cases}
G_1(\mu) \bs{v}+\frac{16}{17} h^2\mu_1\wt{\bs{d}}_{4+}^T\wt{\bs{v}},\quad &j=1,\\
G_j(\mu) \bs v, \quad &j=2,3,4,\ldots.
\end{cases}
\]
Even though the new difference operators use a ghost point, we have not added tilde symbols on
$(\ubar{G}(\mu), \ubar{\bs{b}}_1)$. This is to emphasize that they are modified from the
operators without ghost points, $(G(\mu), \bs{b}_1)$.
% AP: No transpose is needed on b_1

\section{Boundary conditions}\label{sec_bc}
To present the techniques for imposing boundary conditions with and without ghost points, and to
highlight the relation between the SBP-GP and SBP-SAT approaches, we consider the one-dimensional
wave equation,
\begin{align}
  \rho U_{tt} = (\mu(x) U_x)_x,\quad x\in [0,1],\ t\geq 0, \label{wave1d}
\end{align}
subject to smooth initial conditions.
Here, $\rho(x)>0$ and $\mu(x)>0$ are material parameters. The dependent variable $U(x,t)$ could, for
example, represent the acoustic overpressure in a linearized model of a compressible fluid. $U_{tt}$
is the second derivative with respect to time and the subscript $x$ denotes differentiation with
respect to the spatial variable.
% Where is the phase velocity needed?
%The phase velocity in this model satisfies $c=\sqrt{\mu/\rho}$. 

We have for simplicity not included a forcing function in the right-hand side of
\eqref{wave1d}. This is because it has no influence on how boundary conditions are imposed. We only
consider imposing the boundary condition on the left boundary, $x=0$. Consequently, boundary terms
corresponding to the right boundary are omitted from the description below. Furthermore, the initial
conditions are assumed to be compatible with the boundary conditions.

\subsection{Neumann boundary conditions}\label{sec_Neumann}

We start by considering the Neumann boundary condition
\begin{equation}\label{eq_neumann}
  U_x(0,t)=f(t),\quad t\geq 0.
\end{equation}
In the SBP-GP method, the semi-discretization of \eqref{wave1d}-\eqref{eq_neumann} is
\begin{align}
  \bs\rho {\bs u}_{tt} &= \wt{G}(\mu)\wt{\bs u},\quad t\geq 0,\label{wave1d_gp} \\
  \wt{\bs{b}}_1^T \wt{\bs{u}}&=f(t),\quad t\geq 0,\label{wave1d_gp_bc}
\end{align}
where $\bs\rho$ is a diagonal matrix with the $j^{th}$ diagonal element $\rho_{j}=\rho(x_j)$, $\wt{\bs{u}}=\wt{\bs{u}}(t)$ is
a time-dependent grid function on $\wt{\bs x}$ and $\bs{u} = \bs{u}(t)$ is the corresponding grid function on
$\bs{x}$. By using the SBP identity \eqref{SBP_GP}, we obtain
\begin{equation*}
\begin{split}
\left(\bs{u}_t, \bs{\rho} \bs{u}_{tt}\right)_{h} &= \left(\bs{u}_t, \wt{G}(\mu) \wt{\bs{u}}\right)_{h} \\
&= -S_{\mu} (\bs u_t, \bs u) - (u_1)_t \mu_1 \wt{\bs{b}}_1^T \wt{\bs u},
\end{split}
\end{equation*}
which can be written as,
\begin{equation}\label{dedt0}
\left(\bs u_t, \bs{\rho} \bs u_{tt}\right)_{h} + S_{\mu} (\bs u_t, \bs u)  = - (u_1)_t \mu_1 \wt{\bs{b}}_1^T \wt{\bs u}.
\end{equation}

We define the discrete energy 
\begin{equation*}
E_h:=(\bs u_t, \bs{\rho} \bs u_t)_h + S_{\mu} (\bs u, \bs u),
\end{equation*}
and note that the left-hand side of equation \eqref{dedt0} equals the change rate of the discrete energy,
\begin{equation}\label{dedt1}
\f{d}{dt}E_h= - 2(u_1)_t \mu_1 \wt{\bs{b}}_1^T \wt{\bs u}.
\end{equation}
To obtain energy stability, we need to impose the Neumann boundary condition such that the right-hand
side of \eqref{dedt1} is non-positive when $f=0$. The key in the SBP-GP method is to use the ghost point as the
additional degree of freedom for imposing the boundary condition. Here, the Neumann boundary condition
\eqref{eq_neumann} is approximated by enforcing $\wt{\bs{b}}_1^T \wt{\bs{u}}(t)=f(t)$. From \eqref{B1}, it
is satisfied if
\begin{equation}\label{eq_neu_bc}
%\frac{1}{12h}(-3u_{0}-10u_1+18u_2-6u_3+u_4) = 0,\quad\Rightarrow\quad
u_0 = \frac{1}{3}(-10u_1+18u_2-6u_3+u_4 -12h f(t)),\quad t\geq 0.
\end{equation}
This relation gives the ghost point value $u_0$ as function of the interior values $u_j$,
$j=1,2,3,4$. The resulting approximation is energy conservative because
\begin{equation}\label{energy_estimate_Neumann_GP}
\f{d}{dt}E_h=0,\quad f(t)=0.
\end{equation}

Next, consider the semi-discretization of \eqref{wave1d} by the SBP-SAT method in~\cite{Mattsson2012}, 
\begin{equation}\label{wave1d_sat}
\bs\rho \bs u_{tt} =  G(\mu)\bs u + \bs{p}_N,
\end{equation}
where $\bs{p}_N$ is a penalty term for enforcing the Neumann condition \eqref{eq_neumann}. By using
the SBP identity \eqref{SBP_NGP}, we obtain
\begin{equation*}
\begin{split}
(\bs u_t, \bs\rho \bs u_{tt})_h &= (\bs u_t,  G(\mu)\bs u)_{h} +(\bs u_t,\bs{p}_N)_{h} \\
&=- S_{\mu}(\bs u_t,\bs u)  - (u_1)_t \mu_1 \bs{b}_1^T \bs{u} +(\bs u_t,\bs{p}_N)_h, 
\end{split}
\end{equation*}
which can be written as
\begin{equation}\label{energy_estimate_Neumann_SAT0}
\f{d}{dt}[(\bs u_t, \bs\rho \bs u_{t})_h + S_{\mu}(\bs u,\bs u) ]= - 2 (u_1)_t \mu_1 \bs{b}_1^T \bs{u} + 2(\bs u_t, \bs{p}_N)_h. 
\end{equation}
To obtain energy conservation, the right hand side of \eqref{energy_estimate_Neumann_SAT0} must
vanish when $f(t)=0$. This property is satisfied by choosing
\begin{equation}\label{eq_neumann-pen}
  \bs{p}_N=\mu_1 h^{-1} w_1^{-1} \left(\bs{b}_1^T \bs{u} - f(t) \right)\bs{e}_1,
%  \bs{p}_N=\mu_1 h^{-1} w_1^{-1} (\bs{b}_1^T \bs{u}) \bs{e}_1,
\end{equation}
where $\bs{e}_1=[1,0,0,\cdots]^T$. On the boundary, \eqref{wave1d_sat} can therefore be written as
\begin{multline*}
  \rho_{1} (u_1)_{tt} =  G_1(\mu)\bs{u} + \frac{\mu_1}{h w_1} \left( \bs{b}_1^T \bs{u} - f(t) \right)
  = G_1(\mu)\bs{u} + \frac{\mu_1}{h w_1}\left( \ubar{\bs{b}}_1^T\wt{\bs{u}} +
    \f{1}{3} h^3 \wt{\bs{d}}_{4+}\wt{\bs{u}}  - f(t) \right)\\
    = \ubar{G}_1(\mu)\wt{\bs{u}} + \frac{\mu_1}{h w_1} \left( \ubar{\bs{b}}_1^T\wt{\bs{u}}  - f(t) \right),
\end{multline*}
where we have used \eqref{eq_bndryder_wgp} and \eqref{eq_without_with_gp} to express the relations
between SBP operators with and without ghost points. For $j\geq 2$, the penalty term $\bs{p}_N$ is zero and
$\ubar{G}_j(\mu)\wt{\bs{u}} = G_j(\mu)\bs{u}$. Thus, we can write the SBP-SAT discretization as,
\begin{align}
  \bs{\rho} \bs{u}_{tt}  &= \ubar{G}(\mu)\wt{\bs{u}},\quad t\geq 0,\label{wave1d_sat_gp1}\\
  \ubar{\bs{b}}_1^T \wt{\bs{u}}&=f(t),\quad t\geq 0,\label{wave1d_sat_gp2}
\end{align}
which is of the same form as the SBP-GP discretization \eqref{wave1d_gp}-\eqref{wave1d_gp_bc}. Thus,
for Neumann boundary conditions, the SAT penalty method is equivalent with the SBP-GP method. An
interesting consequence is that, if both formulations are integrated in time by the same scheme, 
\eqref{wave1d_sat} and \eqref{wave1d_sat_gp1}-\eqref{wave1d_sat_gp2}
will produce identical solutions. Thus, solutions of the SBP-SAT method will satisfy the
Neumann boundary condition strongly, in the same point-wise manner as the SBP-GP method.

%Because the right hand side of \eqref{energy_estimate_Neumann_SAT0} is identically zero, the
%SBP-SAT discretization is energy conserving.
%
%An obvious choice of the the penalty term is to take $h^{-1}w_1^{-1}\mu_0 \bs{b}_1^T \bs{u} $ as
%the first component of $\bs{p}_N$, and zero elsewhere. This choice leads to an energy conserving
%discretization with the energy estimate
%\begin{equation}\label{energy_estimate_Neumann_SAT}
%\f{d}{dt}[(\bs u_t, \bs\rho \bs u_{t})_h + S_{\mu}(\bs u,\bs u) ]= 0. 
%\end{equation}
%

%Since $\bs{b}_1^T \bs{u}$ is a third order approximation of $\f{du}{dx}(x_1)$, the penalty term
%introduces a truncation error of $\mathcal{O}(h^2)$ at the boundary. This error is of the same order
%as the truncation error of the SBP operator $G(\mu)$ at the boundary. Therefore, the order of
%accuracy in the solution is not affected by the penalty term. Because of the equivalence between the
%methods, the boundary approximation \eqref{B1} used by the SBP-GP method could be replaced by a
%third order approximation. This modification would result in a method with the same order of
%accuracy in the solution. However, for the same grid size, the error in the solution would probably
%be larger with a third order approximation of the boundary condition.

Since $\bs{b}_1^T \bs{u}$ is a third order approximation of $\f{du}{dx}(x_1)$, the penalty term
introduces a truncation error of $\mathcal{O}(h^2)$ at the boundary, that is, $\bs{b}_1^T \bs{u}=\f{du}{dx}(x_1)+\mathcal{O}(h^2)$. This error is of the same order
as the truncation error of the SBP operator $G(\mu)$ at the boundary. Therefore, the order of the 
largest truncation error in the discretization is not affected by the penalty term. Because of the equivalence between the
methods, the boundary approximation \eqref{B1} used by the SBP-GP method could be replaced by a
third order approximation. This modification would result in a method with the same order of
truncation error in the discretization.  

%
% AP: I just proved that the SBP-SAT and SBP-GP methods are
% equivalent for Neumann bc. If you discretize time with the same method, they will therefore give
% identical solutions
%
%However, $\wt{\bs{b}}_1^T
%\wt{\bs{u}} =0$ is satisfied at every time step, but $\bs{b}_1^T \bs{u}$ is only approximately zero.
%An inhomogeneous Neumann boundary condition, $U_x(0,t)=f(t)$, where $f(t)$ is a smooth function, can
%be imposed in the same way.  In the SBP-SAT method, the penalty term \eqref{eq_neumann-pen} is
%replaced by $\bs{p}_N=\mu_1 h^{-1} w_1^{-1} (\bs{b}_1^T \bs{u} - f(t))\bs{e}_1$.
%In the SBP-GP method, we replace \eqref{wave1d_gp_bc} by $\wt{\bs{b}}_1^T\wt{\bs{u}}=f(t)$.

\subsection{Dirichlet boundary conditions}\label{sec_Dirichlet}

Consider the wave equation \eqref{wave1d} subject to the Dirichlet boundary condition,
\begin{equation}\label{eq_dirichlet}
  U(0,t)=g(t),\quad t\geq 0.
\end{equation}
The most obvious way of discretizing \eqref{eq_dirichlet} would be to set $u_1=g(t)$ for all
times. However, that condition is not directly applicable for the SBP-GP method because it does not involve
the ghost point value $u_0$. Instead, we can differentiate \eqref{eq_dirichlet} twice with respect to
time and use \eqref{wave1d_gp} to approximate $U_{tt}(0,t)=g_{tt}(t)$,
\begin{equation}\label{eq_semi-discrete-dir}
(u_1)_{tt} = \frac{1}{\rho_{1}}\wt{G}_1(\mu)\wt{\bs u}=g_{tt}(t),\quad t\geq 0.
\end{equation}
From \eqref{eq_sbp-gp-1}, the above condition is satisfied if the ghost point value is related to the interior values according to
\begin{equation}\label{eq_discrete_dir}
 u_0 = \frac{17}{12\mu_1} \left(h^2 \rho_{1}g_{tt}(t) -\sum_{k=1}^8 \sum_{m=1}^8 \beta_{k,m}\mu_m u_k\right),\quad t\geq 0.
 \end{equation}
This relation corresponds to \eqref{eq_neu_bc} for Neumann boundary conditions.
Because the initial conditions are compatible with the boundary condition, we can integrate
\eqref{eq_semi-discrete-dir} once in time to get $(u_1)_t=g_t(t)$. Therefore,  when $g_t=0$, the approximation is energy
conserving because the right hand side of \eqref{dedt1} vanishes and the solution satisfies \eqref{energy_estimate_Neumann_GP}.

Because we impose the Dirichlet condition through \eqref{eq_semi-discrete-dir}, we see that
\eqref{wave1d_gp} is equivalent to
\[
\left.\bs\rho {\bs u}_{tt}\right|_j =
\begin{cases}
  \rho_{1} g_{tt},&\quad j=1,\\
  \wt{G}_j(\mu)\wt{\bs u},&\quad j=2,3,\ldots.
\end{cases}
\]
Since the ghost point value is only used by $\wt{G}$ on the boundary itself, this approximation is
independent of the ghost point value and can be interpreted as injection of the Dirichlet
data, $u_1(t) = g(t)$, and the energy stability follows. Injection can also be used to impose Dirichlet data for the SBP operators without ghost
point. Here, energy stability can be proved from a different perspective by analyzing the properties
of the matrix representing the operator $G(\mu)$, see~\cite{Duru2014}.  While the injection approach
provides the most straightforward way of imposing Dirichlet data, it does not generalize to the
interface problem.

For SBP operators without ghost points, it is also possible to impose a Dirichlet boundary condition
by the SAT penalty method. In this case, the penalty term has a more complicated form than in the
Neumann case, but the technique sheds light on how to impose grid interface conditions. Replacing
the penalty term in \eqref{wave1d_sat} by $\bs{p}_D$, an analogue of the energy rate equation
\eqref{energy_estimate_Neumann_SAT0} is
\begin{equation}\label{energy_estimate_Dirichlet_SAT0}
\f{d}{dt}[(\bs u_t, \bs{\rho} \bs{u}_{t})_{h} +  S_{\mu}(\bs{u}, \bs{u})]= - 2  \mu_1 (u_1)_t \bs{b}_1^T \bs{u}  + 2(\bs u_t,\bs{p}_D)_{h}. 
\end{equation}

It is not straightforward to choose $\bs{p}_D$ such that the right-hand side of
\eqref{energy_estimate_Dirichlet_SAT0} is non-positive.  However, we can choose $\bs{p}_D$ so that
the right-hand side of \eqref{energy_estimate_Dirichlet_SAT0} becomes part of the energy. For example, if
\begin{equation}\label{p_d}
%\bs{p}_D = (hw_1)^{-1} \mu_1 \bs b_1 u_1+\f{\tau}{h}(hw_1)^{-1} \mu_1\bs e_1 u_1
%\bs{p}_D = (hw_1)^{-1} \mu_1 (\bs b_1+\f{\tau}{h}\bs e_1) u_1
\bs{p}_D =  -\mu_1 (u_1-g(t) ) W^{-1} (\bs b_1+\f{\tau}{h}\bs e_1),
\end{equation}
where $\bs e_1=[1,0,0,\cdots]^T$ and $W$ is the diagonal SBP norm matrix. With homogeneous boundary condition $g(t)=0$, we have
\begin{equation*}
(\bs{u}_t,\bs{p}_D)_{h} = - \mu_1 u_1 \bs{b}_1^T \bs{u}_t - \f{\tau}{h} \mu_1 (u_1)_t  u_1,
\end{equation*}
and \eqref{energy_estimate_Dirichlet_SAT0} becomes
\begin{equation}\label{energy_estimate_Dirichlet_SAT1}
\f{d}{dt}\left[(\bs u_t, \bs\rho \bs u_{t})_{h} +  S_{\mu}(\bs u, \bs u) + 2\mu_1 u_{1} \bs{b}_1^T \bs{u}+\f{\tau}{h}\mu_1 u_1^2
\right]= 0. 
\end{equation}
We obtain an energy estimate if the quantity in the square bracket is non-negative.
% WHAT DOES THIS MEAN:
%Note that the penalty term \eqref{p_d} does not
%introduce truncation error with the Dirichlet boundary condition $u(x_1)=0$. 

In Lemma 2 of \cite{Virta2014}, it is proved that the following identity holds
\begin{equation}\label{borrow}
S_{\mu}(\bs u, \bs u) =  \ubar S_{\mu}(\bs u, \bs u)  + h\alpha \mu_{\min} (\bs{b}_1^T \bs{u})^2,
\end{equation}
where both the bilinear forms $S_{\mu}(\cdot,\cdot)$ and $\ubar S_{\mu}(\cdot,\cdot)$ are symmetric
and positive semi-definite, $\alpha$ is a constant that depends on the order of accuracy of $G(\mu)$
but not $h$, and
\[
\mu_{\min} = \min_{1\leq j \leq r}\mu_j.
\]
The integer constant $r$ depends on the order of accuracy of $G(\mu)$ but not on $h$. As an example,
the fourth order accurate SBP operator $G(\mu)$ constructed in \cite{Mattsson2012} satisfies
\eqref{borrow} with $r=4$ and $\alpha=0.2505765857$. Any $\alpha>0.2505765857$ can make $\ubar
S_{\mu}(\cdot,\cdot)$ indefinite. Identities corresponding to \eqref{borrow} have been used in several
other SBP related methodologies, e.g.~\cite{Appelo2007,Duru2014V,Mattsson2009}.

By using \eqref{borrow},
\begin{multline*}
  S_{\mu}(\bs u, \bs u)+2\mu_1 u_{1} \bs{b}_1^T \bs{u} + \f{\tau}{h}\mu_1 u_1^2
  = \ubar S_{\mu}(\bs u, \bs u) + h\alpha \mu_{\min} (\bs{b}_1^T \bs{u})^2+2\mu_1 u_{1} \bs{b}_1^T \bs{u}+\f{\tau}{h}\mu_1 u_1^2\\
  = \ubar S_{\mu}(\bs u, \bs u) +\left( \sqrt{h\alpha \mu_{\min}}(\bs{b}_1^T\bs{u}) +
  \frac{1}{\sqrt{h\alpha \mu_{\min}}}\mu_1 u_1\right)^2 - \frac{1}{h\alpha \mu_{\min}}\mu_1^2 u_1^2+\f{\tau}{h}\mu_1 u_1^2\\
  =\ubar S_{\mu}(\bs u, \bs u) +\left( \sqrt{h\alpha \mu_{\min}}(\bs{b}_1^T \bs{u}) + \frac{1}{\sqrt{h\alpha \mu_{\min}}}\mu_1
  u_1\right)^2+\left(\frac{\tau}{h}\mu_1-\frac{\mu_1^2}{h\alpha\mu_{\min}} \right) u_1^2.
\end{multline*}
Thus, the quantity in the square bracket of  \eqref{energy_estimate_Dirichlet_SAT1} is an energy if,
\[
\frac{\tau}{h}\mu_1-\frac{\mu_1^2}{h\alpha\mu_{\min}}\geq 0\quad \Rightarrow\quad \tau\geq \frac{\mu_1}{\alpha\mu_{\min}}.
\]
We note that the penalty parameter $\tau$ has a lower bound but no upper bound. Choosing $\tau$ to
be equal to the lower bound gives large numerical error in the solution \cite{Wang2017a}. However,
an unnecessarily large $\tau$ causes stiffness and leads to stability restrictions on the
time-step~\cite{Mattsson2009}. In computations, we find that increasing $\tau$ by 10\% to 20\% from
the lower bound is a good compromise for accuracy and efficiency.

The energy estimate \eqref{energy_estimate_Dirichlet_SAT1} contains two more terms than the
corresponding estimate for the SBP-GP method. The additional terms are approximately zero up to the order of
accuracy because of the Dirichlet boundary condition $u(x_1)=0$.

\subsection{Time discretization with the SBP-GP method}

Let $\wt{\bs{u}}^k$ denote the numerical approximation of $U(\wt{\bs{x}}, t_k)$,
where $t_k=k\delta_t$ for $k=0,1,2,\ldots$ and $\delta_t>0$ is the constant time step. We start by
discussing the update procedure for the explicit Str\"omer scheme, which is second order accurate
in time. For simplicity we only consider the boundary conditions at $x=0$. The time-stepping
procedure is described in Algorithm \ref{alg1}.
\begin{algorithm}
\caption{Second order accurate time stepping with ghost points for Neumann or Dirichlet boundary
  conditions.}\label{alg1}
\smallskip
Given initial conditions $\wt{\bs{u}}^0$ and $\wt{\bs{u}}^{-1}$ that satisfy the discretized boundary
conditions.
\begin{enumerate}
  \item Update the solution at all interior grid points,
  \begin{equation}\label{eq_wave1d_gp_full}
    {\bs{u}^{k+1}} = 2\bs{u}^k-\bs{u}^{k-1} + {\delta_t^2} \bs{\rho}^{-1}
    \wt{G}(\mu)\wt{\bs{u}}^k,\quad k=0,1,2,\ldots
  \end{equation}
\item[2a.] For Neumann boundary conditions, assign the ghost point value $u_0^{k+1}$ to satisfy
  \begin{equation}\label{eq_wave1d_gp_neu_bc}
    \wt{\bs{b}}_1^T \wt{\bs{u}}^{k+1} = f(t_{k+1}).
  \end{equation}
  \item[2b.] For Dirichlet boundary conditions, assign the ghost point value $u_0^{k+1}$ to satisfy
    \begin{equation}\label{eq_wave1d_gp_dir_bc}
      \wt{G}_1(\mu)\wt{\bs{u}}^{k+1} = \frac{\rho_{1}}{\delta_t^2} (g(t_{k+2}) - 2u_1^{k+1} + u_1^{k}).
    \end{equation}
\end{enumerate}
\end{algorithm}

For Neumann conditions, it is clear that \eqref{eq_wave1d_gp_neu_bc} enforces the semi-discrete
boundary condition \eqref{wave1d_gp_bc} at each time level. This condition must also be satisfied by
the initial data, $\wt{\bs{u}}^0$.

For Dirichlet conditions, we proceed by explaining how \eqref{eq_wave1d_gp_dir_bc} is
related to the semi-discrete boundary condition \eqref{eq_semi-discrete-dir}. Assume that the
initial data satisfies the Dirichlet boundary conditions, that is, $u^0_1 = g(t_0)$ and
$u^{-1}_1=g(t_{-1})$. Also assume that \eqref{eq_wave1d_gp_dir_bc} is satisfied for $\wt{\bs{u}}^0$,
\[
  \wt{G}_1(\mu)\wt{\bs{u}}^{0} = \frac{\rho_{1}}{\delta_t^2} (g(t_{1}) - 2u_1^{0} + u_1^{-1}) =
  \frac{\rho_{1}}{\delta_t^2} (g(t_{1}) - 2 g(t_{0}) + g(t_{-1})).
\]
The solution at time level $t_1$ is obtained from \eqref{eq_wave1d_gp_full}. In particular, on the
boundary,
\[
  {u}^{1}_1 = 2{u}^0_1 - {u}^{-1}_1 + \frac{\delta_t^2}{\rho_1} \wt{G}_1(\mu)\wt{\bs{u}}^0 = 2 g(t_0)
  - g(t_{-1}) +  \frac{\delta_t^2}{\rho_1} \frac{\rho_{1}}{\delta_t^2} \left(g(t_{1}) - 2 g(t_{0}) + g(t_{-1})\right) = g(t_1).
\]
Thus, the Dirichlet boundary condition is also satisfied at time level $t_1$. Assigning the ghost
% AP removed "in" before "such"
point $u^1_0$ such that \eqref{eq_wave1d_gp_dir_bc} is satisfied for $\wt{\bs{u}}^1$ thus ensures
that $\wt{\bs{u}}^2$ will satisfy the Dirichlet boundary condition at the next time level, after
\eqref{eq_wave1d_gp_full} has been applied. By induction, the Dirichlet boundary condition will be
satisfied for any time level $t_k$. The boundary condition \eqref{eq_wave1d_gp_dir_bc} is therefore
equivalent to
\[
\wt{G}_1(\mu)\wt{\bs{u}}^{k+1} = \rho_{1}\frac{g(t_{k+2}) - 2g(t_{k+1}) + g(t_{k})}{\delta_t^2},
\]
which is a second order accurate approximation of the semi-discrete boundary condition
\eqref{eq_semi-discrete-dir}. Another interpretation of \eqref{eq_wave1d_gp_dir_bc} is that the
ghost point value for $\wt{\bs{u}}^{k+1}$ is assigned by ``looking ahead'', i.e., such that the
Dirichlet boundary condition will be satisfied for $\wt{\bs{u}}^{k+2}$.

The Str\"omer time-stepping scheme can be improved to fourth (or higher) order accuracy in time by a
modified equation approach~\cite{Gilbert2008, Sjogreen2012}. To derive the scheme, we first notice that
\begin{equation}\label{eq_fourth-in-time}
\frac{\bs{u}^{k+1} - 2 \bs{u}^k + \bs{u}^{k-1}}{\delta_t^2} = \bs{u}_{tt}(t_k) + \frac{\delta_t^2}{12} \bs{u}_{tttt}(t_k)
+ {\cal O}(\delta_t^4).
\end{equation}
By differentiating \eqref{wave1d_gp} twice in time,
\begin{equation}\label{eq_fourth-derivative}
\bs{u}_{tttt} = \bs{\rho}^{-1} \wt{G}(\mu)\wt{\bs u}_{tt}.
\end{equation}
We can obtain a second order (in time) approximation of $\wt{\bs{u}}_{tt}$ from
\begin{equation}\label{eq_acceleration}
\wt{\bs{v}}^{k} := \frac{\wt{\bs{u}}^{*,k+1} - 2 \wt{\bs{u}}^k + \wt{\bs{u}}^{k-1}}{\delta_t^2} =
\wt{\bs u}_{tt} + {\cal O}(\delta_t^2).
\end{equation}
Here, $\wt{\bs{u}}^{*,k+1}$ is the second order (in time) predictor,
\begin{equation}\label{eq_predictor}
\bs{u}^{*,k+1} = 2\bs{u}^k-\bs{u}^{k-1} + {\delta_t^2} \bs{\rho}^{-1}\wt{G}(\mu)\wt{\bs{u}}^k,
\end{equation}
augmented by appropriate boundary conditions that define the ghost point value ${u}^{*,k+1}_0$. By using
\eqref{eq_acceleration} and \eqref{eq_fourth-derivative} to approximate $\bs{u}_{tttt}$ in \eqref{eq_fourth-in-time},
we obtain
\begin{equation}\label{eq_fourth-order-complete}
\bs{u}^{k+1} = 2 \bs{u}^k - \bs{u}^{k-1} + \delta_t^2\bs{\rho}^{-1} \wt{G}(\mu)\wt{\bs u}^k + \frac{\delta_t^4}{12} \bs{\rho}^{-1} \wt{G}(\mu)\wt{\bs{v}}^{k},
\end{equation}
where $\wt{\bs{v}}^{k}$ is given by \eqref{eq_acceleration}. By subtracting \eqref{eq_predictor}
from \eqref{eq_fourth-order-complete} and re-organizing the terms, we arrive at the corrector formula,
\[
\bs{u}^{k+1} = \bs{u}^{*,k+1} + \frac{\delta_t^4}{12} \bs{\rho}^{-1} \wt{G}(\mu)\wt{\bs{v}}^{k}.
\]
The resulting fourth order predictor-corrector time-stepping procedure is described in Algorithm
\ref{alg2}.

Similar to the second order algorithm, it is straightforward to impose Neumann boundary
conditions, but the Dirichlet boundary conditions require some further explanation. The basic idea
is to enforce the same boundary condition for both the predictor and the corrector, i.e.,
\[
u^{*,k}_1 = u^{k}_1 = g(t_k),\quad k=0,1,2,\ldots.
\]
As before, the Dirichlet condition are enforced by ``looking ahead''.  We assume that the initial
data satisfies the compatibility conditions $u_1^{-1} = g(t_{-1})$,  $u_1^{0} = g(t_{0})$ and
\[
\wt{G}_1(\mu)\wt{\bs{u}}^0 = \rho_1 \frac{ g(t_1) - 2 g(t_{0}) -  g(t_{-1})}{\delta_t^2}.
\]
Similar to the second order time-stepping algorithm, the first predictor step updates the solution on the
boundary to be
\[
u^{*,1}_1 = 2{u}^0_1 - {u}^{-1}_1 + \frac{\delta_t^2}{\rho_1} \wt{G}_1(\mu)\wt{\bs{u}}^0 = 2 g(t_0)
- g(t_{-1}) +  \frac{\delta_t^2}{\rho_1} \frac{\rho_{1}}{\delta_t^2} \left(g(t_{1}) - 2 g(t_{0}) + g(t_{-1})\right) = g(t_1).
\]
Thus, the compatibility condition for the initial condition $\wt{\bs{u}}^0$ ensures that the first
predictor satisfies the Dirichlet boundary condition $u^{*,1}_1 = g(t_1)$. The boundary condition
for the predictor \eqref{eq_pred_bc} assigns the ghost point value $u^{*,1}_0$ such that
\[
\wt{G}_1(\mu)\wt{\bs{u}}^{*,1} =
2 \wt{G}_1(\mu)\wt{\bs{u}}^{0} - \wt{G}_1(\mu)\wt{\bs{u}}^{-1}\quad \Rightarrow\quad
\wt{G}_1(\mu)\wt{\bs{v}}^{0} = 0.
\]
As a result, the corrector formula \eqref{eq_corr}, evaluated at the boundary point, gives
\[
u^{1}_1 = u^{*,1}_1 + \frac{\delta_t^4}{12\rho_1} \wt{G}_1(\mu)\wt{\bs{v}}^{0} = g(t_1).
\]
This shows that both the predictor and the corrector satisfy the Dirichlet boundary condition after
the first time step. By enforcing the boundary condition \eqref{eq_corr_bc} for the corrector, we
guarantee that the next predictor satisfy the Dirichlet boundary condition after \eqref{eq_pred} has
been applied.  An induction argument shows that the Dirichlet conditions are satisfied for all
subsequent time steps.

Both the second order Str\"omer scheme and the fourth order predictor-corrector schemes are stable
under a CFL condition on the time step. Furthermore, the time-discrete solution satisfies an energy
estimate, see~\cite{Kreiss2002, Sjogreen2012} for details.

\begin{algorithm}
\caption{Fourth order accurate predictor-corrector time stepping with ghost points for Neumann or Dirichlet boundary
  conditions.}\label{alg2}
\smallskip
Given initial conditions $\wt{\bs{u}}^0$ and $\wt{\bs{u}}^{-1}$ that satisfy the discretized
boundary conditions. Compute $\wt{\bs{u}}^{*, k+1}$ and $\wt{\bs{u}}^{k+1}$ for $k=0,1,2,\dots$ according to
\begin{enumerate}
\item[1.] Compute the predictor at the interior grid points,
  \begin{equation}\label{eq_pred}
    \bs{u}^{*,k+1} = 2\bs{u}^k-\bs{u}^{k-1} + {\delta_t^2} \bs{\rho}^{-1}
    \wt{G}(\mu)\wt{\bs{u}}^k.
  \end{equation}
\item[2a.] For Neumann boundary conditions, assign the ghost point value $u_0^{*, k+1}$ to satisfy
  \begin{equation}
    \wt{\bs{b}}_1^T \wt{\bs{u}}^{*, k+1} = f(t_{k+1}).
  \end{equation}
\item[2b.] For Dirichlet boundary conditions, assign the ghost point value $u_0^{*, k+1}$ to satisfy
  \begin{equation}\label{eq_pred_bc}
    \wt{G}_1(\mu)\wt{\bs{u}}^{*,k+1} = 2 \wt{G}_1(\mu)\wt{\bs{u}}^{k} - \wt{G}_1(\mu)\wt{\bs{u}}^{k-1}.
  \end{equation}
\item[3.] Evaluate the acceleration at all grid points,
  \begin{equation}
    \wt{\bs{v}}^{k} := \frac{\wt{\bs{u}}^{*,k+1} - 2 \wt{\bs{u}}^k + \wt{\bs{u}}^{k-1}}{\delta_t^2}.
  \end{equation}
\item[4.] Compute the corrector at the interior grid points,
  \begin{equation}\label{eq_corr}
    \bs{u}^{k+1} = \bs{u}^{*,k+1} + \frac{\delta_t^4}{12} \bs{\rho}^{-1} \wt{G}(\mu)\wt{\bs{v}}^{k}.
  \end{equation}
\item[5a.] For Neumann boundary conditions, assign the ghost point value $u_0^{k+1}$ to satisfy
  \begin{equation}
    \wt{\bs{b}}_1^T \wt{\bs{u}}^{k+1} = f(t_{k+1}).
  \end{equation}
\item[5b.] For Dirichlet boundary conditions, assign the ghost point value $u_0^{k+1}$ to satisfy
  \begin{equation}\label{eq_corr_bc}
    \wt{G}_1(\mu)\wt{\bs{u}}^{k+1} = \frac{\rho_{1}}{\delta_t^2} (g(t_{k+2}) - 2u_1^{k+1} + u_1^{k}).
  \end{equation}
\end{enumerate}
\end{algorithm}

%Algorithm~\ref{alg1} can be used to solve the Dirichlet problem by modifying the
%second step to assign the ghost point value according to \eqref{eq_discrete_dir}.

\section{Grid refinement interface}\label{sec_mr}
%We consider the wave equation in two space dimensions with a discontinuous wave speed. To achieve high order accuracy with a finite difference method, it is important that the difference stencil does not cross the discontinuity. A common strategy for discontinuous parameters is to partition the domain into subdomains, and align the discontinuity with the subdomain boundaries. The finite difference approximation is then carried out in each subdomain, and adjacent subdomains are connected via interface conditions. 
To obtain high order accuracy at a material discontinuity, we partition the domain into subdomains
such that the discontinuity is aligned with a subdomain boundary. The multiblock finite difference
approximation is then carried out in each subdomain where the material is smooth, and adjacent
subdomains are connected by interface conditions.

%To achieve high order accuracy when the material property is discontinuous, it is important that the difference stencil does not cross the discontinuity. A common strategy for discontinuous parameters is to partition the domain into subdomains, and align the discontinuity with the subdomain boundaries. The finite difference approximation is then carried out in each subdomain, and adjacent subdomains are connected via interface conditions. 

As an example, we consider the two-dimensional acoustic wave equation in a composite domain
$\Omega^f\cup\Omega^c$, where $\Omega^f=[0,1]\times [0,1]$ and $\Omega^c=[0,1]\times [-1,0]$. The
governing equation in terms of the acoustic pressure can be written as
\begin{equation}\label{wave2d}
\begin{split}
&\rho^f F_{tt} = \nabla\cdot (\mu^f\nabla F),\ (x,y)\in\Omega^f, \ t\geq 0, \\
&\rho^c C_{tt} = \nabla\cdot (\mu^c\nabla C),\ (x,y)\in\Omega^c, \ t\geq 0,
%&\rho_f \f{d^2}{dt^2}U_f = \nabla\cdot \mu_f\nabla U_f,\ (x,y)\in\Omega_f, \ t\geq 0, \\
%&\rho_c \f{d^2}{dt^2}U_c = \nabla\cdot \mu_c\nabla U_c,\ (x,y)\in\Omega_c, \ t\geq 0,
\end{split}
\end{equation}
with suitable initial and boundary conditions. We assume that the material properties $\mu^f$ and
$\rho^f$ are smooth in $\Omega^f$, and $\mu^c$ and $\rho^c$ are smooth in
$\Omega^c$. However, the material properties may not vary smoothly across the interface between
$\Omega^f$ and $\Omega^c$.

We consider the case where the interface conditions prescribe continuity of pressure and continuity of
normal flux~\cite{Graff1991}:
\begin{equation}\label{interface_condition_2d}
\begin{split}
F(x,0,t) &= C(x,0,t), \\
\mu^f(x,0) \frac{\p F}{\p y}(x,0,t) &= \mu^c(x,0) \frac{\p C}{\p y}(x,0,t),
\end{split}
\quad 0\leq x\leq 1,\quad t\geq 0.
\end{equation}
With the above set of interface conditions, the acoustic energy is conserved across the interface
\cite{Mattsson2008,Petersson2010}.

If the wave speeds are different in the two subdomains, for computational efficiency, different grid
spacings are desirable so that the number of grid points per wavelength becomes the same in both
subdomains~\cite{Hagstrom2012,Kreiss1972}. This leads to a mesh refinement interface with hanging
nodes along $y=0$. Special care is therefore needed to couple the solutions along the interface. In
the following, we consider a grid interface with mesh refinement ratio 1:2, and focus on the
numerical treatment of the interface conditions \eqref{interface_condition_2d}. Other ratios can be
treated analogously.

For simplicity, we consider periodic boundary conditions in $x$. For the spatial discretization, we use
a Cartesian mesh with mesh size $h$ in the (fine) domain $\Omega^f$ and $2h$ in the (coarse) domain
$\Omega^c$, see Figure~\ref{xfxc}.
\begin{figure}
\centering
\includegraphics[width=0.7\textwidth,trim={0 4cm 0 4cm},clip]{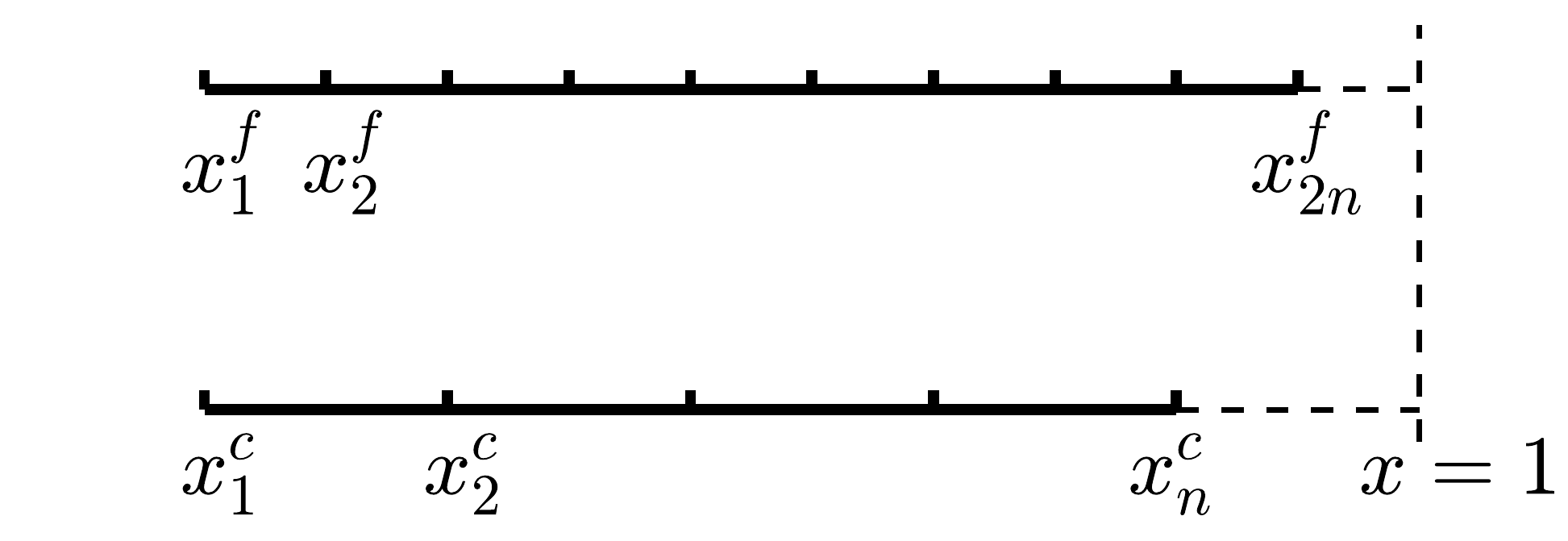}
\caption{A sketch of the grids $\bs x^f$ and $\bs x^c$.}
\label{xfxc}
\end{figure}
The number of grid points in the $x$ direction is $n$ in $\Omega^c$, and $2n$ in
$\Omega^f$, where $h=1/(2n)$. We have excluded grid points on the periodic boundary $x=1$, because the
solution at $x=1$ is the same as at $x=0$. The grid points $(\bs{x}^f, \bs{y}^f)$ in $\Omega^f$ and
$(\bs{x}^c,\bs{y}^c)$ in $\Omega^c$ are defined as
\begin{equation}\label{mesh}
\begin{cases}
x^f_i = (i-1)h,\quad i = 1,2,\cdots, 2n,\\
y^f_j = (j-1)h, \quad j = 0,1,2,\cdots,2n+1
\end{cases}
\text{and}\quad
\begin{cases}
x^c_i = 2(i-1)h,\quad i = 1,2,\cdots, n,\\
y^c_j = 2(j-n)h, \quad j = 0,1,2,\cdots,n+1
\end{cases},
\end{equation}
respectively. There are $2n$ ghost points 
\begin{equation}\label{gpf}
(x^f_i,y^f_0),\ i=1,2,\cdots,2n
\end{equation}
in $\Omega^f$ and $n$ ghost points 
\begin{equation}\label{gpc}
(x^c_i,y^c_{n+1}),\ i=1,2,\cdots,n
\end{equation}
in $\Omega^c$.
%We have excluded values on the boundary $x=1$, because we do not solve them in the
%numerical scheme. Instead, the numerical solution at $x=1$ is set to be equal to the numerical
%solution at $x=0$ because of the periodic boundary condition.

%If the wave speed in $\Omega^c$ is twice as large as in $\Omega^f$, then the number of grid points per wavelength is constant in the entire
%domain, which is ideal for computational efficiency. However, this leads to a mesh refinement interface with hanging nodes along the interface
%$y=0$. In the following, we discuss both the SBP-GP and SBP-SAT method with energy conserving
%interpolation for the mesh refinement interface.

Notations for the two-dimensional SBP operators are introduced in Section \ref{sec_2d}.  The
SBP-GP method for the problem \eqref{wave2d}-\eqref{interface_condition_2d} is introduced in
Section~\ref{sec_original}. A second order accurate method was originally developed
in~\cite{Petersson2010}, where ghost points from both subdomains are used to impose the interface
conditions. Here, we generalize the technique to fourth order accuracy. In Section~\ref{sec_new}, we
propose a new SBP-GP method that only uses ghost points from the coarse domain. This reduces the
amount of computational work for calculating the numerical solution at the ghost points and improves
the structure of the associated linear system. We end this section with a discussion of the SBP-SAT
method and its relation to the SBP-GP method.

\subsection{SBP identities in two space dimensions}\label{sec_2d}
%A finite difference approximation of a two dimensional differential operator can be achieved dimension by dimension. As can be seen from \eqref{SBP_GP} and \eqref{SBP_NGP}, the SBP identities of the second derivative GP-SBP operator and NGP-SBP operator are in the same form. Therefore, in the discussion of SBP properties in two space dimensions, we use the notations of SBP operators with ghost point.

The one-dimensional SBP identities with ghost points \eqref{SBP_GP} are on exactly the same form as
those without ghost points, \eqref{SBP_NGP}. In the discussion of SBP identities in two space
dimensions, we use the notations for SBP operators with ghost points in $\Omega^f$. The same
notational convention of the tilde symbol is used to indicate that the corresponding variable uses
ghost points.

Let $\bs u$ and $\bs v$ be grid functions in $\Omega^f$. We define the two-dimensional scalar
product
\[
(\bs u, \bs v)_h = h^2 \sum_{i=1}^{2n}\sum_{j=1}^{2n-1} w_j u_{ij} v_{ij}.
\]
The weights $w_j$ do not depend on the index $i$ because of the periodic boundary condition in
$x$. In addition, we define the scalar product for grid functions on the interface
\begin{equation}\label{ip_g}
  \langle\bs{u}_{\Gamma}, \bs{v}_{\Gamma}\rangle_{h}= h \sum_{i=1}^{2n} u_{i} v_{i},
  %, \quad \langle\bs{p}, \bs{q}\rangle_{2h} = 2h \sum_{i=1}^{n} p_{i} q_{i}.  
\end{equation}
where the subscript $\Gamma$ denotes the grid function on the interface. 
%where the superscript $\bs\Gamma$ denotes grid functions on the interface. 
%The SBP identity in two space dimensions in the fine domain $\Omega^f$ is
%\begin{equation}\label{SBP_GP_2D}
%(\bs u, \wt G_f(\mu) \wt{\bs v})_{h} = -S_f(\bs u, \bs v) - \langle \bs{u},\bs{v_{b}}\rangle_{h}. 
%\end{equation}

The SBP identity in two space dimensions in the fine domain $\Omega^f$ can be written as
\begin{align}
(\bs u, G_x(\mu) \bs v)_{h} &= -S_x(\bs u, \bs v), \label{SBP_GP_2Dx} \\
(\bs u, \wt G_y(\mu) \wt{\bs v})_{h} &= -S_y(\bs u, \bs v) - \langle\bs u_{\Gamma}, \wt{\bs  v}'_{\Gamma}\rangle_h, \label{SBP_GP_2Dy}
\end{align}
where the subscripts $x$ and $y$ denote the spatial direction that the operator acts on. The
bilinear forms $S_x(\cdot, \cdot)$ and $S_y(\cdot, \cdot)$ are symmetric and positive
semi-definite. There is no boundary term in \eqref{SBP_GP_2Dx} for $G_x(\mu)$ because of the
periodic boundary condition. For simplicity, we have omitted the boundary term from the boundary at
$y=1$. The last term on the right hand side of \eqref{SBP_GP_2Dy} corresponds to the boundary term
from the interface, where the $i^{th}$ element of $\wt{\bs v}'_{\Gamma}$ is
\begin{equation}\label{v_elementwise}
(v'_{\Gamma})_i =  \mu_{i,1}^f \wt{\bs{b}}_1^T \wt{\bs{v}}_{i,:}.% =: \sum_{j=0}^4 \sigma_j v_{ij},
%B_f(\bs u, \bs v) = h \sum_{i=1}^{2n} u_{i1} \mu_{i,1}^f \sum_{j=0}^4 \sigma_j v_{ij}
\end{equation}
%where $\sigma_j$ are the coefficients in $\wt{\bs{b}}_1^T$, given in \eqref{B1}, multiplied by
%$\mu_{i,1}^f$. 
Here we use Matlab's colon notation, i.e., $:$ denotes all grid points in the corresponding index direction. 
%Note that the tilde
%symbol on $\wt{\bs v}'_{\Gamma}$ indicates that ghost points are used to compute its quantity, but
%$\wt{\bs v}'_{\Gamma}$ itself is not defined at the ghost points.
%\begin{equation}\label{def_nabla}
%q_i=\mu_{i,1}^f\sum_{j=0}^4 \sigma_j v_{ij},\quad i=1,2,\cdots,2n,
%\end{equation}
%with the coefficients $\sigma_j$ given in \eqref{B1}. %We use the colon symbol to indicate that $\wt{\bs v}_{i,:}$ is a grid function on $(x_i,y_j),\ j=0,1,\cdots,2n-2$.

To condense notation, we define 
\[
\wt{G}_f (\mu)=  G_x(\mu)+\wt{G}_y(\mu),\quad S_f = S_x+S_y,
\]
so that \eqref{SBP_GP_2Dx}-\eqref{SBP_GP_2Dy} can be written as
\begin{equation}\label{SBP_GP_2D}
(\bs u, \wt G_f(\mu) \wt{\bs v})_{h} = -S_f(\bs u, \bs v) - \langle\bs u_{\Gamma}, \wt{\bs v}'_{\Gamma}\rangle_h.
\end{equation}
%In addition, the following identity follows from the symmetry of $S_f(\cdot,\cdot)$,
%\begin{equation}\label{SBP_GP_2D_symmetry}
%(\bs u, \wt G_f(\mu) \wt{\bs v})_{h} = (\wt G_f(\mu) \wt{\bs u}, {\bs v})_h - \langle\bs u_{\Gamma}, \bs v'_{\Gamma}\rangle_h + \langle\bs v_{\Gamma}, \bs u'_{\Gamma}\rangle_h. 
%\end{equation}
The SBP identity for the operators in the coarse domain $\Omega^c$ are defined similarly.

\subsection{The fourth order accurate SBP-GP method}\label{sec_original}
We approximate \eqref{wave2d} by
\begin{align}
\bs{\rho^f} \bs f_{tt} &= \wt G_f(\mu) \wt{\bs f}, \label{semi_original_f}\\
\bs{\rho^c} \bs c_{tt} &= \wt G_c(\mu) \wt{\bs c}, \label{semi_original_c}
\end{align}
where the grid functions $\bs f$ and $\bs c$ are finite difference approximations of the functions
$F(x,y,t)$ and $C(x,y,t)$ in \eqref{wave2d}, respectively. The diagonal matrices $\bs{\rho}^f$ and
$\bs{\rho}^c$ contain the material properties $\rho^f$ and $\rho^c$ evaluated on the fine and coarse
grids, respectively. Corresponding to the continuous interface condition
\eqref{interface_condition_2d}, the grid functions $\bs f$ and $\bs c$ are coupled through the
discrete interface conditions
\begin{align}
\bs{f}_{\Gamma} &=\mathcal{P}\bs c_{\Gamma},  \label{Operator_I} \\
\wt{\bs c}'_{\Gamma}&= \mathcal{R} \wt{\bs f}'_{\Gamma}.   \label{Operator_R}
\end{align}
Here, $\mathcal{P}$ is an operator that interpolates a coarse interface grid function to an
interface grid function on the fine grid. The operator $\mathcal{R}$ performs the opposite
operation. It restricts an interface grid function on the fine grid to the coarse
grid. Stability of the difference approximation relies on the compatibility between the operators
$\mathcal{P}$ and $\mathcal{R}$, as is specified in the following theorem.
\begin{theorem}
The semi-discretization \eqref{semi_original_f}-\eqref{Operator_R} satisfies the energy estimate
\begin{equation}\label{energy_rate_0}
\frac{d}{dt}\left[(\bs f_t, \bs{\rho^f}\bs f_t)_h + S_f(\bs f,\bs f)+(\bs c_t, \bs{\rho^c}\bs c_t)_{2h} + S_c(\bs c,\bs c) \right] = 0,
\end{equation}
if the interpolation and restriction operators are compatible,
\begin{equation}\label{P2RT}
\mathcal{P}=2\mathcal{R}^T.
\end{equation}
\end{theorem}

\begin{proof}
By using the SBP identity \eqref{SBP_GP_2D} in $\Omega^f$, we obtain
\[
(\bs f_t, \bs{\rho^f} \bs f_{tt})_h + S_f(\bs f_t, \bs f) =  -\left\langle\left(\bs f_{\Gamma}\right)_t, \wt{\bs f}'_{\Gamma}\right\rangle_h.
\]
Similarly, we have in $\Omega^c$
\[
(\bs c_t, \bs{\rho^c} \bs c_{tt})_{2h} + S_c(\bs c_t, \bs c) =  \left\langle\left(\bs c_{\Gamma}\right)_t, \wt{\bs c}'_{\Gamma}\right\rangle_{2h}.
\]
Summing the above two equations yields
\begin{equation}\label{proof_temp1}
\frac{d}{dt}\left[(\bs f_t, \bs{\rho^f}\bs f_t)_h + S_f(\bs f,\bs f)+(\bs c_t, \bs{\rho^c}\bs c_t)_{2h}
  + S_c(\bs c,\bs c) \right] = -2\left\langle\left(\bs f_{\Gamma}\right)_t, \wt{\bs f}'_{\Gamma}\right\rangle_h
+ 2\left\langle\left(\bs c_{\Gamma}\right)_t, \wt{\bs c}'_{\Gamma}\right\rangle_{2h}.
\end{equation}
To prove that the right-hand side vanishes, we first differentiate \eqref{Operator_I} in time, and use \eqref{ip_g} to obtain 
\[
\left\langle\left(\bs f_{\Gamma}\right)_t, \wt{\bs f}'_{\Gamma}\right\rangle_h = \left\langle\left(\mathcal{P}\bs
c_{\Gamma}\right)_t, \wt{\bs f}'_{\Gamma}\right\rangle_h.
\]
The compatibility condition \eqref{P2RT}, together with the scalar product \eqref{ip_g}, gives
\[
\left\langle\left(\mathcal{P}\bs c_{\Gamma}\right)_t, \wt{\bs f}'_{\Gamma}\right\rangle_h = \left\langle\left(\bs
c_{\Gamma}\right)_t, \mathcal{R}\wt{\bs f}'_{\Gamma}\right\rangle_{2h}.
\]
The second interface condition \eqref{Operator_R} leads to 
\begin{equation}\label{proof_temp3}
\left\langle\left(\bs c_{\Gamma}\right)_t, \mathcal{R}\wt{\bs f}'_{\Gamma}\right\rangle_{2h} = \left\langle\left(\bs
c_{\Gamma}\right)_t, \wt{\bs c}'_{\Gamma}\right\rangle_{2h}.
\end{equation}
The energy rate relation \eqref{energy_rate_0} follows by inserting \eqref{proof_temp3} into the
right hand side of \eqref{proof_temp1}. This proves the theorem.
\end{proof}

% AP reformulated the sentence
We note that the factor 2 in the compatibility condition \eqref{P2RT} arises because of the 1:2 mesh
refinement ratio in two dimension and the periodic boundary condition. The factor is 4 in the
corresponding three dimensional case.
% AP: The following statement does not make any sense to me: P and R act on grid functions, not norms
%For a more general case, the compatibility condition can be
%written as $H_P \mathcal{P}=(H_R \mathcal{R})^T$, where $H_P$ and $H_R$ are the discrete norms
%corresponding to the grids on the interface.

For the mesh refinement ratio 1:2, the stencils in $\mathcal{P}$ and $\mathcal{R}$ can be easily
computed by a Taylor series expansion. For example, a fourth order interpolation operator in
\eqref{Operator_I} has the stencil
\begin{align*}
(f_{\Gamma})_{2i} &= -\frac{1}{16}(c_{\Gamma})_{i-1}+\frac{9}{16}(c_{\Gamma})_{i}+\frac{9}{16}(c_{\Gamma})_{i+1}-\frac{1}{16}(c_{\Gamma})_{i+2},\\
(f_{\Gamma})_{2i-1} &= (c_{\Gamma})_{i}
\end{align*}
on the hanging and coinciding nodes, respectively. Then, the compatibility condition \eqref{P2RT}
determines the restriction operator $\mathcal{R}$, used by the second interface condition \eqref{Operator_R},
\[
(c'_{\Gamma})_{i} =
-\frac{1}{32}(f'_{\Gamma})_{2i-4}+\frac{9}{32}(f'_{\Gamma})_{2i-2}+\frac{1}{2}(f'_{\Gamma})_{2i-1}+\frac{9}{32}(f'_{\Gamma})_{2i}-\frac{1}{32}(f'_{\Gamma})_{2i+2}.
\]
For other mesh refinement ratios, the interpolation and restriction operators can be constructed
using the techniques in \cite{Kozdon2016}.

Similar to Dirichlet boundary conditions for the one-dimensional problem, ghost points are
not explicitly involved in the first interface condition \eqref{Operator_I}. However, by
differentiating \eqref{Operator_I} twice in time and using the semi-discretized equations
\eqref{semi_original_f}-\eqref{semi_original_c}, we obtain
\begin{equation}\label{Operator_Itt}
\left.\left(\bs{\rho^f}\right)^{-1} \wt G_f(\mu) \wt{\bs f}\right|_{\Gamma} =
\mathcal{P}\left(\left.\left(\bs{\rho^c}\right)^{-1} \wt G_c(\mu) \wt{\bs c}\right|_{\Gamma}\right).
\end{equation}
This condition depends on the ghost point values on both sides of the interface and is equivalent to
\eqref{Operator_I} if the initial data also satisfies that condition. For this reason, we impose
interface conditions for the semi-discrete problem through \eqref{Operator_R} and
\eqref{Operator_Itt}. When discretizing \eqref{semi_original_f}-\eqref{semi_original_c} in time by
the predictor-corrector method, the fully discrete time-stepping method follows by the same principle as
the predictor-corrector method in Algorithm~\ref{alg2}. More precisely, for the predictor, step 2a is used to enforce
\eqref{Operator_R} and step 2b is used for \eqref{Operator_I}. Similarily, for the corrector, step
5a is used to enforce \eqref{Operator_R}, combined with step 5b for \eqref{Operator_I}.

The grid function $\wt{\bs c}'_{\Gamma}$ in \eqref{Operator_R} has $n$ elements. By writing
\eqref{Operator_R} in element-wise form it becomes clear that it is a system of $n$ linear equations
that depends on $3n$ unknown ghost point values. Similarly, \eqref{Operator_Itt} is a system of $2n$ linear
equations for the same $3n$ unknowns. In combination, the two interface conditions give a system of $3n$
linear equations, whose solution determines the $3n$ ghost point values. For the fully discrete
problem, this linear system must be solved once during the predictor step and once during the
corrector step.

The coefficients in the linear equations are independent of time. As a consequence, an efficient
solution strategy is to LU-factorize the interface system once, before the time stepping
starts. Backward substitution can then be used to calculate the ghost point values during the
time-stepping. For problems in three space dimensions, computations are performed on many processors
on a parallel distributed memory machine. Then it may not be straightforward to efficently calculate
the LU-factorization. As an alternative, iterative solvers can be used. For example, an iterative
block Jacobi relaxation method is used in \cite{Petersson2010}. It has proven to work well in
practice for large-scale problems.
% AP reformulated the last sentence

\subsection{The improved SBP-GP method}\label{sec_new}
\begin{figure}
\includegraphics[width=.8\textwidth,trim={0 6cm 0 7cm},clip]{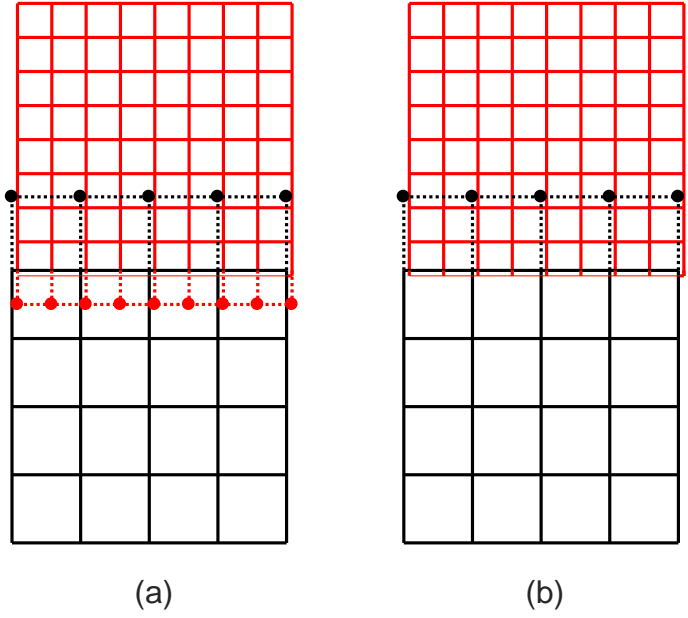}
\caption{A mesh refinement interface with ghost points denoted by filled circles. (a) ghost points
  from both domains. (b) ghost points from the coarse domain.}
\label{mesh_gp_12}
\end{figure}
In the improved SBP-GP method, the interface conditions are imposed through $n$ linear equations that
only depend on the $n$ ghost point values in $\wt{\bs{c}}$, see Figure~\ref{mesh_gp_12}~(b).
%
%In the fourth order accurate SBP-GP method presented in Section \ref{sec_original}, $n-1$ ghost
%points from the coarse domain $\Omega^c$ and $2n-2$ ghost points from the fine domain $\Omega^f$ are
%used to impose interface conditions. As a consequence, we need to solve two systems of linear
%equations whose coefficients are independent of time. In this section, we present an improved SBP-GP
%method, where only $n-1$ ghost points from $\Omega^c$ are used for interface conditions. This
%reduces the number of linear equations to $n-1$.
%
The key to the improved method is to combine SBP operators with and without ghost points. More
precisely, in $\Omega^c$ we use the SBP operator with ghost points. Thus, the semi-discretized
equation in $\Omega^c$ is the same as in the original SBP-GP method,
\begin{equation}\label{semi_new_c}
\bs{\rho^c}\bs c_{tt} = \wt G_c(\mu)\bs {\wt c}.
\end{equation}
In $\Omega^f$, we use \eqref{semi_original_f} only for the grid points that are not on the interface
\begin{equation}\label{eqn_f2}
(\bs{\rho^f}\bs f_{tt})_{:,j}  = (G_f(\mu) \bs f)_{:,j},\quad j=2,3,\ldots.
\end{equation}
%Here we use Matlab's colon notation, i.e., $:$ denotes all grid points in the corresponding index direction. 
For the grid points in $\Omega^f$ that are on the interface, we enforce the interface condition \eqref{Operator_I} such that 
\begin{equation}\label{eqn_f1}
\bs{f}_{:,1} = \mathcal{P}(\bs{c}_{:,n}). % AP added parentesis
\end{equation}
Note that this equation does not depend on any ghost point values in $\Omega^f$.

To write the semi-discretization in a compact form and prepare for the energy analysis, we
differentiate \eqref{eqn_f1} twice in time, and use \eqref{semi_new_c} to obtain
\begin{equation}\label{eqn_f3}
\left(\bs f_{tt}\right)_{:,1} = \mathcal{P}\left((\bs c_{tt})_{:,n}\right) = % AP Moved left parenthesis before "P"
\mathcal{P}\left(\left.\left(\bs{\rho}^c\right)^{-1}
\wt{G}_c(\mu)\wt{\bs{c}}\right|_{\Gamma}\right). 
\end{equation}
Equations \eqref{eqn_f2} and \eqref{eqn_f3} can be combined into
\begin{equation}\label{semi_new_f}
\left(\bs{\rho^f} \bs f_{tt}\right)_{:,j} := \left( L_h\bs f\right)_{:,j} =
\begin{cases}
\left(G_f(\mu) \bs f\right)_{:,1}+ \bs\eta_{:},& j=1,\\
\left(G_f(\mu) \bs f\right)_{:,j},& j=2,3,\ldots,
\end{cases}
\end{equation}
where 
\begin{equation*}
\bs\eta = \bs{\rho^f}|_{\Gamma}\mathcal{P}\left(\left.\left(\bs{\rho^c}\right)^{-1} \wt G_c(\mu)\bs {\wt c}\right|_{\Gamma}\right)
 - \left.G_f(\mu) \bs f\right|_{\Gamma}.
\end{equation*}
We note that $\bs\eta$ is a zero vector up to truncation errors in the SBP operator and the
interpolation operator. Therefore, $\bs\eta$ does not affect the order of accuracy in the spatial
discretization.

The semi-discretization \eqref{semi_new_c} and \eqref{semi_new_f} can be viewed as a hybridization
of the SBP-GP method and the SBP-SAT method. The spatial discretization \eqref{semi_new_f} in
$\Omega^f$ is on the SBP-SAT form, but the penalty term $\bs\eta$ depends on  the
ghost points values in $\wt{\bs c}$.

Continuity of the solution is imposed by \eqref{eqn_f1}, in the same way as in the original SBP-GP
method. But to account for the contribution from $\bs\eta$, continuity of flux (the second interface
condition in \eqref{interface_condition_2d}) must be imposed differently. Here we use
\begin{equation}\label{Rcondition}
\wt{\bs c}'_{\Gamma}= \mathcal{R} \left(\bs f'_{\Gamma}-hw_1\bs\eta\right),
\end{equation}
where $h$ is the mesh size in $\Omega_f$, and $w_1$ is the first entry in the scalar product
\eqref{wip}. Note that ghost points are used to compute $\wt{\bs c}'_{\Gamma}$ but not $\bs f'_{\Gamma}$.

Compared with \eqref{Operator_R} in the original SBP-GP method, the condition \eqref{Rcondition}
includes the term $h w_1 \bs\eta$. Because it is on the order of the truncation error it does not
affect the order of accuracy. As a consequence, \eqref{Rcondition} provides a valid way of enforcing
flux continuity. The following theorem illustrates why the $\bs{\eta}$-term is important for energy
stability.
\begin{theorem}
Assume that the interpolation and restriction operators satify \eqref{P2RT}. Then, the semi-discrete approximation
\eqref{semi_new_c}, \eqref{semi_new_f} and \eqref{Rcondition} is energy stable in the sense that
\eqref{energy_rate_0} holds.
\end{theorem}
\begin{proof}
From \eqref{semi_new_f} , we have 
\begin{align*}
\left(\bs f_t, \bs{\rho^f}\bs f_{tt}\right)_h & = \left(\bs f_t, G_f(\mu) \bs f\right)_h + hw_1 \left\langle
  \bs{f}_t|_{\Gamma}, \bs{\eta} \right\rangle_h \\
& = -S_f(\bs f_t, \bs f) - \left\langle  \bs{f}_t|_{\Gamma}, \bs f'_{\Gamma}\right\rangle_h+ hw_1 \left\langle  \bs{f}_t|_{\Gamma}, \bs\eta\right\rangle_h\\
& = -S_f(\bs f_t, \bs f) + \left\langle  \bs{f}_t|_{\Gamma}, -\bs f'_{\Gamma}+hw_1\bs\eta\right\rangle_h.
\end{align*}
The contribution from the domain $\Omega^c$ is 
\[
(\bs c_t, \bs{\rho^c}\bs c_{tt})_{2h}  = -S_c(\bs c_t, \bs c) + \langle  \bs{c}_t|_{\Gamma}, \wt{\bs c}'_{\Gamma}\rangle_{2h}.
\]
Adding the two above equations gives
\begin{multline*}
\frac{d}{dt}\left[(\bs f_t, \bs{\rho^f}\bs f_t)_h + S_f(\bs f,\bs f)+(\bs c_t, \bs{\rho^c}\bs
  c_t)_{2h} + S_c(\bs c,\bs c) \right]
= 2\langle  \bs{f}_t|_{\Gamma}, -\bs f'_{\Gamma}+hw_1\bs\eta\rangle_h + 2\left\langle
\bs{c}_t|_{\Gamma}, \wt{\bs c}'_{\Gamma}\right\rangle_{2h} \\
 = 2\langle  \mathcal{P}\bs{c}_t|_{\Gamma}, -\bs f'_{\Gamma}+hw_1\bs\eta\rangle_h + 2\langle
 \bs{c}_t|_{\Gamma}, \wt{\bs c}'_{\Gamma}\rangle_{2h} % AP added tilde over c'
 = 2\langle  \bs{c}_t|_{\Gamma}, \mathcal{R}(-\bs f'_{\Gamma}+hw_1\bs\eta)\rangle_{2h} + 2\langle
 \bs{c}_t|_{\Gamma}, \wt{\bs c}'_{\Gamma}\rangle_{2h}  = 0. % AP added tilde over c'
\end{multline*}
\end{proof}

With the predictor-corrector method for the time discretization of \eqref{semi_new_c} and
\eqref{semi_new_f}, the fully discrete algorithm can be adopted from Algorithm \ref{alg2}. We impose
\eqref{Rcondition} in step 2a for the predictor, and in step 5a for the corrector. We note that
\eqref{Rcondition} corresponds to a system of $n$ linear equations. The right-hand sides are
different in the linear systems in steps 2a and 5a, but the matrix is the same. It can therefore be
LU-factorized once, before time integration starts. The linear systems can then be solved by
backward substitution during the time stepping. The improved SBP-GP method presented in this section
is evaluated through numerical experiments in Section~\ref{sec_num}.

\subsection{The SBP-SAT method}
%With stable SBP-SAT schemes for both the Neumann problem in Section \ref{sec_Neumann} and the Dirichlet problem in Section \ref{sec_Dirichlet}, 
%it is straightforward to derive the penalty terms for the interface conditions \eqref{interface_condition_2d}. The semi-discretization can be written
In the SBP-SAT method, the penalty terms for the interface conditions \eqref{interface_condition_2d} can be constructed by combining the penalty terms for the Neumann problem in Section \ref{sec_Neumann} and the Dirichlet problem in Section \ref{sec_Dirichlet}. The semi-discretization can be written as
\begin{align}
\bs{\rho^f} \bs f_{tt} &= G_f(\mu) \bs f+ \bs{p_f},  \label{SAT_INTf}\\
\bs{\rho^c} \bs c_{tt} &= G_c(\mu) \bs{c} + \bs{p_c}. \label{SAT_INTc}
\end{align}

There are two choices of $\bs{p_f}$ and $\bs{p_c}$. The first version, developed in \cite{Wang2016}, uses three penalty terms
\begin{align}
(p_f)_{i,:} = W^{-1}_f \left[ -\mu_{i,1}^f \f{1}{2}\bs b_1^f \left(\bs{f}_{\Gamma}-\mathcal{P} \bs{c}_{\Gamma}\right)_i  -\mu_{i,1}^f \f{\tau_f}{h}\bs e_1^f \left(\bs{f}_{\Gamma}-\mathcal{P} \bs{c}_{\Gamma}\right)_i       +    \f{1}{2} \bs e_1^f\left(\bs{f}'_{\Gamma}-\mathcal{P} \bs{c}'_{\Gamma}\right)_i\right], \label{SATf}\\
(p_c)_{i,:} =  W^{-1}_c \left[ -\mu_{i,1}^c\f{1}{2}\bs b_1^c \left(\bs{c}_{\Gamma}-\mathcal{R} \bs{f}_{\Gamma}\right)_i   -\mu_{i,1}^c\f{\tau_c}{2h}\bs e_1^c \left(\bs{c}_{\Gamma}-\mathcal{R} \bs{f}_{\Gamma}\right)_i      +    \f{1}{2} \bs e_1^c\left(\bs{c}'_{\Gamma}-\mathcal{R} \bs{f}'_{\Gamma}\right)_i\right], \label{SATc}
\end{align}
where $\bs b_1$ and $\bs e_1$ act in the $y$ direction. 
%
%\begin{equation}\label{SATf}
%(\bs f_t, \bs{p_f})_h = -\frac{1}{2}\langle(\bs{f_{\nabla}^{\Gamma}})_t, \bs{f^{\Gamma}}-\mathcal{P}\bs{c^{\Gamma}}\rangle_h - \frac{\tau_f}{h} \langle \bs f^{\bs\Gamma}_t, \bs{f^{\Gamma}}-\mathcal{P}\bs{c^{\Gamma}} \rangle_{h}  + \frac{1}{2} \langle(\bs{f^{\Gamma}})_t,\bs{f_{\nabla}^{\Gamma}}-\mathcal{P} \bs{c_{\nabla}^{\Gamma}} \rangle_h,
%\end{equation}
%and 
%\begin{equation}\label{SATc}
%(\bs c_t, \bs{p_c})_{2h} =  \frac{1}{2}\langle (\bs{c_{\nabla}^{\Gamma}})_t, \bs{c^{\Gamma}}-\mathcal{R}\bs{f^{\Gamma}}\rangle_{2h} - \frac{\tau_c}{2h} \langle \bs c^{\bs\Gamma}_t, \bs{c^{\Gamma}}-\mathcal{R}\bs{f^{\Gamma}} \rangle_{2h}  - \frac{1}{2} \langle (\bs{c^{\Gamma}})_t,\bs{c_{\nabla}^{\Gamma}}-\mathcal{R} \bs{f_{\nabla}^{\Gamma}} \rangle_{2h}.
%\end{equation}
In both \eqref{SATf} and \eqref{SATc}, the first two terms  penalize continuity of the solution, and the third term penalizes continuity of the flux. %Energy stability is proved in \cite{Wang2016} for the special case when $\mu$ is constant. Following the same approach, we find that t
The scheme \eqref{SAT_INTf}-\eqref{SATc} is energy stable when the penalty parameters satisfy 
\begin{equation}\label{tau_int}
\tau_f = \frac{1}{2}\tau_c \geq \max_{i,j}\left(\frac{(\mu^f_{i,1})^2}{2(\mu_{\min}^f)_i\alpha}, \frac{(\mu^c_{j,n})^2}{2(\mu_{\min}^c)_j\alpha}\right),
\end{equation}
where $i=1,2,\ldots, 2n$ and $j=1,2,\ldots,n$. 

The second choice of SATs uses four penalty terms \cite{Wang2018}, which has a better stability
property for problems with curved interfaces.  The method was improved further in
\cite{Almquist2019} from the accuracy perspective when non-periodic boundary conditions are used in
the $x$-direction. In addition, the penalty parameters in \cite{Almquist2019} are optimized and 
% AP removed comma before "and"
are sharper than those in \cite{Wang2018}. % AP edited sentence
As will be seen in the numerical experiments, the sharper penalty
parameters lead to an improved CFL condition. % AP edited

\subsection{Computational complexity}
In the next section, we test numerically the CFL condition of the improved SBP-GP method and the
SBP-SAT method for cases with a grid refinement interface. To enable a fair comparison in terms of
computational efficiency, in this section we estimate the computational cost of the two methods for
one time step. Since the interior stencils of the two SBP operators are the same, the main
difference in computational cost comes from how the interface conditions % AP simplified sentence
are imposed at each time step. For simplicity, we only consider problems with constant coefficients
when estimating the computational complexity. % AP edited
Also note that the number of floating point operations (flops) stated below depends on the
implementation of the algorithms, and should not be considered exact.

% AP: removed inappropriate "O" in several places and reorganized the text
In the improved SBP-GP method, a system of $n$ linear equations must be solved at each time step,
where $n$ is the number of grid points on the interface in the coarse domain. The system matrix is
banded with bandwidth 7, so the LU factorization requires $49n$ flops,
but it is only computed once before the time stepping begins. In each time step, updating the right hand
side of the linear system and solving by backward substitution requires $173n$ and $5n$ flops,
respectively. This results in a grand total of $178n$ flops at each time step.

% AP reformulated this paragraph
In the SBP-SAT method, the interface conditions are imposed by the SAT terms, which are updated at
each time step. This calculation requires $157n$ flops. We conclude that imposing interface
conditions with the SBP-GP and the SBP-SAT method require a comparable number of floating point
operations per time step. Thus, the main difference in computational efficiency comes from the different
CFL stability restrictions on the time step, which is investigated in the following section.

\section{Numerical experiments}\label{sec_num}
In this section, we conduct numerical experiments to compare the SBP-GP method and the SBP-SAT method in terms of computational efficiency. Our first focus is CFL condition, which is an important factor in solving large-scale problems. We numerically test the effect of different boundary and interface techniques on the CFL condition with the predictor-corrector time stepping method. We then compare $L^2$ error and convergence rate of the SBP-GP method and the SBP-SAT method with the same spatial and temporal discretizations. The convergence rate is computed by
\[\log\left(\frac{e_h}{e_{2h}}\right)\bigg/ \log\left(\frac{1}{2}\right),\]
where $e_{2h}$ is the $L^2$ error on a grid $\bs x$, and $e_{h}$ is the $L^2$ error on a grid with grid size half of $\bs x$ in each subdomain and spatial direction. 

\subsection{Time-stepping stability restrictions}\label{sec_cfl_analysis}
We consider the scalar wave equation in one space dimension
\begin{equation}\label{1dwave_ne}
\rho U_{tt} =  (\mu U_x)_x + F,
\end{equation}
in the domain  $x\in [-\pi/2,\pi/2]$ with non-periodic boundary conditions. %The manufactured smooth solution is chosen as 
%$u = \cos(x+2t),$
%and is  used to obtain initial and boundary data, and the forcing function. 

%We discretize equation \eqref{1dwave_ne} by using the fourth order accurate SBP operator, and use a
%predictor-corrector time stepping method \cite{Sjogreen2012} for the time integration.

%In general, we do not have a closed form expression of the CFL condition for the finite difference discretization. Instead, we can estimate
%the CFL condition by considering periodic boundary conditions and Fourier methods. In this case, the 
%standard fourth order accurate finite difference stencil
%\[
%(u_{xx})_j\approx -\frac{1}{12}u_{j+2}+\frac{4}{3}u_{j+1}-\frac{5}{2}u_j+\frac{4}{3}u_{j-1}-\frac{1}{12}u_{j-2}
%\]
%can be used on every grid point. Let $u_j=\hat u e^{i\omega x_j}$ and $F=0$, we find
%\[
%\hat u_{tt} = \left(-\frac{1}{12}e^{i\omega x_{j+2}}+\frac{4}{3}e^{i\omega x_{j+1}}-\frac{5}{2}e^{i\omega x_{j}}+\frac{4}{3}e^{i\omega x_{j-1}}-\frac{1}{12}e^{i\omega x_{j-2}}\right)\hat u
%\]
%\begin{lemma}\label{Fourier}
%w
%\end{lemma}

%More precisely,
%the Fourier transform of the fourth order accurate central finite difference stencil is
%\begin{equation*}
%\widehat Q = -\frac{4}{h^2}\sin^2\frac{\omega h}{2}\left(1+\frac{1}{3}\sin^2\frac{\omega h}{2}\right),
%\end{equation*}
%where $\omega$ is the wave number and $h$ is the grid size \cite[pp.~ 9]{Gustafsson2008}.
In \cite{Sjogreen2012}, it is proved that for the predictor-corrector time stepping method, the time
step constraint by the CFL condition is
\begin{equation}\label{deltat_1d}
\delta_t\leq \frac{2\sqrt{3}}{\sqrt{\kappa}},
\end{equation}
where $\kappa$ is the spectral radius of the spatial discretization matrix.  In general, we do not
have a closed form expression for $\kappa$. In the special case of periodic boundary conditions and
constant coefficients, $\kappa$ is given by the following lemma.

\begin{lemma}\label{lemmaQ} % AP reformulated
Consider \eqref{1dwave_ne} with periodic boundary conditions, constant $\rho$, $\mu$ and zero
forcing $F=0$. If the equation is discretized with standard fourth order accurate centered finite
differences, the spectral radius becomes
\[
\kappa = \frac{16\mu}{3h^2\rho},
\]
where $h$ is the grid spacing. 
\end{lemma}

\begin{proof}
See Appendix 1.
\end{proof}

In the following numerical experiments, we choose $\rho=\mu=1$, which gives the estimated CFL condition 
$\delta_t\leq 1.5h$. This case is used below as a reference when comparing CFL conditions. % AP edited

First, we consider the Neumann boundary condition at $x=\pm \pi/2$, and use the SBP-GP and the
SBP-SAT method to solve the equation \eqref{1dwave_ne} until $t=200$. For the SBP-GP method with the
fourth order SBP operator derived in \cite{Sjogreen2012}, we find that the scheme is stable when
$\delta_t\leq 1.44h$. In other words, the time step needs to be reduced by about $4\%$ when
comparing with the reference CFL condition. For the SBP-SAT method with the fourth order SBP
operator derived in \cite{Mattsson2004}, the scheme is stable up to the reference CFL condition
$\delta_t\leq 1.5h$.

Next, we consider the equation with Dirichlet boundary conditions at $x=\pm \pi/2$. To test the
injection method and the SAT method, we use the fourth order accurate SBP operator without ghost
point \cite{Mattsson2004}. When using the injection method to impose the Dirichlet boundary
condition, the scheme is stable with $\delta_t \leq 1.5h$. However, when using the SAT method to
weakly impose the Dirichlet boundary condition and choosing the penalty parameter $20\%$ larger than
its stability-limiting value, the scheme is only stable if $\delta_t \leq 1.16h$. This amounts to a
% AP edited
reduction in time step by $23\%$. If we decrease the penalty parameter so that it is only $0.1\%$
larger than its stability-limiting value, then the scheme is stable with $\delta_t \leq 1.25h$,
i.e. the time step needs to be reduced by $17\%$, compared to the injection method. %AP edited

In conclusion, for the Neumann boundary condition, both the SBP-GP and the SBP-SAT method can be
used with a time step comparable to that given by the reference CFL condition. This is not
surprising, given the similarity of the methods and in the discrete energy expressions. For the Dirichlet
boundary condition, we need to reduce the time step by $23\%$ in the SAT method. If we instead
inject the Dirichlet data, then the scheme is stable with the time step given by the reference CFL
condition. % derived from Fourier analysis for the periodic boundary problem. %AP: already said this
           % at the top

\subsection{Discontinuous material properties}\label{ex_Snell}
We now investigate the SBP-GP and SBP-SAT method for the wave equation with a mesh refinement
interface. The model problem is
\begin{equation}\label{eqn}
\rho U_{tt} = \nabla\cdot(\mu\nabla U)+F, % AP: added comma
\end{equation}
in a two-dimensional domain $\Omega=[0,4\pi]\times[-4\pi,4\pi]$, where $\rho(x,y)>0$, $\mu(x,y)>0$,
and the wave speed is $c=\sqrt{\mu/\rho}$. Equation \eqref{eqn} is augmented with Dirichlet boundary
conditions at $y=\pm 4\pi$, and periodic boundary conditions at $x=0$ and $x=4\pi$.

The domain $\Omega$ is divided into two subdomains $\Omega^1=[0,4\pi]\times [-4\pi,0]$ and
$\Omega^2=[0,4\pi]\times [0,4\pi]$ with an interface $\Gamma$ at $y=0$. The material parameter $\mu$
is a smooth function in each subdomain, but may be discontinuous across the interface. In
particular, we consider two cases: $\mu$ is piecewise constant in Section \ref{ex_Snell}, and $\mu$
is a smooth function in Section \ref{ex_smooth}. In each case, we test the fourth order accurate
SBP-GP method and the SBP-SAT method, both in terms of the CFL condition and the convergence rate. % AP

When $\mu$ is piecewise constant, an analytical solution can be constructed by Snell's law.
%Let the computational domain be $\Omega=[-4\pi,4\pi]\times [0,4\pi]$, consisting of two square
%shaped subdomains $\Omega_l=[-4\pi,0]\times [0,4\pi]$ and $\Omega_r=[0,4\pi]\times [0,4\pi]$ with
%an interface $\Gamma=\Omega_l\cap\Omega_r$ at $x=0$.  
We choose a unit density $\rho=1$ and denote the piecewise constant $\mu$ as
\begin{equation*}
\mu(x,y)=\begin{cases}
\mu_1,\quad & (x,y)\in\Omega^1,\\
\mu_2,\quad & (x,y)\in\Omega^2,
\end{cases}
\end{equation*}
where $\mu_1\neq \mu_2$. 

\begin{figure}
\centering
\includegraphics[width=.5\textwidth]{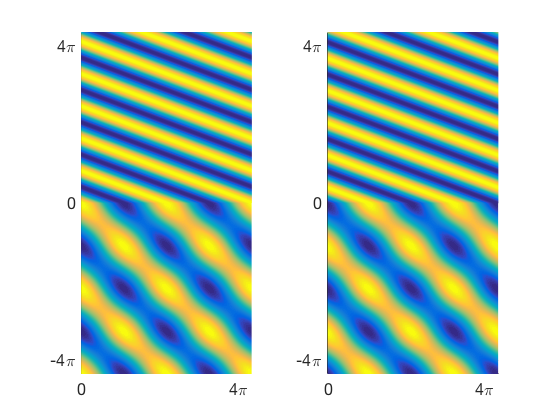}
\caption{The exact solution at time $t=0$ (left), and $t=11$ (right) when the wave has propagated for about 2.5 temporal periods. The solution is continuous at the material interface $x=0$ but the normal derivative is discontinuous due to the material discontinuity.}
\label{fig_exact_sol}
\end{figure}

Let an incoming plane wave $U_I$ travel in $\Omega^1$ and impinge on the interface $\Gamma$. The
resulting field consists of the incoming wave $U_I$, as well as a reflected field $U_R$ and a
transmitted field $U_T$. With the ansatz
\begin{equation*}
\begin{split}
& U_I = \cos(x+y-\sqrt{2\mu_1} t), \\
& U_R = R\cos(-x+y+\sqrt{2\mu_1} t), \\
& U_T = T\cos(x+ky-\sqrt{2\mu_1} t),
\end{split}
\end{equation*}
where $k=\sqrt{2\mu_1/\mu_2-1}$, the two parameters $R$ and $T$ are determined by the interface conditions
\begin{equation*}
\begin{split}
 U_I+U_R&=U_T,\\
{\mu_1} \frac{\partial}{\partial x} (U_I+U_R)&={\mu_2} \frac{\partial}{\partial x} U_T,
\end{split}
\end{equation*}
yielding $R = (\mu_1-\mu_2 k)/(\mu_1+\mu_2 k)$ and $T=1+R$.

In the following experiments, we choose $\mu_1=1$ and $\mu_2=0.25$. As a consequence, the wave speed
is $c_1=1$ in $\Omega^1$ and $c_2=0.5$ in $\Omega^2$. To keep the number of grid points per
wavelength the same in two subdomains, we use a coarse grid with grid spacing $2h$ in $\Omega^1$,
and a fine grid with grid spacing $h$ in $\Omega^2$. We let the wave propagate from $t=0$ until
$t=11$. The exact solution at these two points in time are shown in Figure \ref{fig_exact_sol}. % AP

\subsubsection{CFL condition}
To derive an estimated CFL condition, we perform a Fourier analysis in each subdomain $\Omega^1$ and $\Omega^2$. Assuming periodicity in both spatial directions, the spectral radius of the spatial discretization in $\Omega^1$ and $\Omega^2$ is the same $\kappa = 4/(3h^2)$, given by Lemma \ref{lemmaQ}. By using \eqref{deltat_1d}, we find that the estimated CFL condition is 
\begin{equation}\label{deltat_2d}
\delta_t\leq \frac{1}{\sqrt{2}}\frac{2\sqrt{3}}{\sqrt{4/(3h^2)}}=\frac{3}{\sqrt{2}}h\approx 2.12h.
\end{equation}
We note that  the restriction on time step is the same in both subdomains. The factor $1/\sqrt{2}$ in \eqref{deltat_2d}, which is not present in \eqref{deltat_1d}, comes from \eqref{eqn} having two space dimensions. 

%Assume that in equation \eqref{eqn}, the density $\rho=1$ and $\mu$ is a constant in $\Omega$, and the boundary conditions are periodic in both $x$ and $y$. With the second order central finite difference method in space and time, the von Neumann analysis shows that stability requires $\delta_t\leq \delta_t^{2nd,Von}= \delta_x/\sqrt{2\mu}$, where $\delta_x$ is the grid size in $x$ and $y$. In our experiments, we have a perfect match between the grid size and material parameter: $\delta_x=2h$, $\mu=1$ in $\Omega_1$, and $\delta_x=h$, $\mu=0.25$ in $\Omega_2$, leading to $\delta_t^{2nd, Von}=\sqrt{2}h$ in both subdomains. As a consequence, the number of points per wavelength is uniform in $\Omega$, and the restriction on time step is the same in both subdomains. For the fourth order method, it is shown in \cite{Sjogreen2012} that the stability limit on time step is $\delta_t^{4th, Von}\approx1.46\delta_t^{2nd, Von}$. 

For the SBP-GP method, we have found numerically that the method is stable when the time step $\delta_t\leq 2.09h$. This indicates that the non-periodic boundary condition and the non-conforming grid interface do not affect time step restriction of the SBP-GP method. With $\delta_t=2.09h$ and $641^2$ grid points in the coarse domain, we perform a long time simulation until $t=1000$, and plot the $L_2$ error in Figure \ref{LongTimeSimulation}. We observe that the $L_2$ error does not grow in time, which verifies that the discretization is stable. 

\begin{figure}
\centering
\includegraphics[width=.65\textwidth]{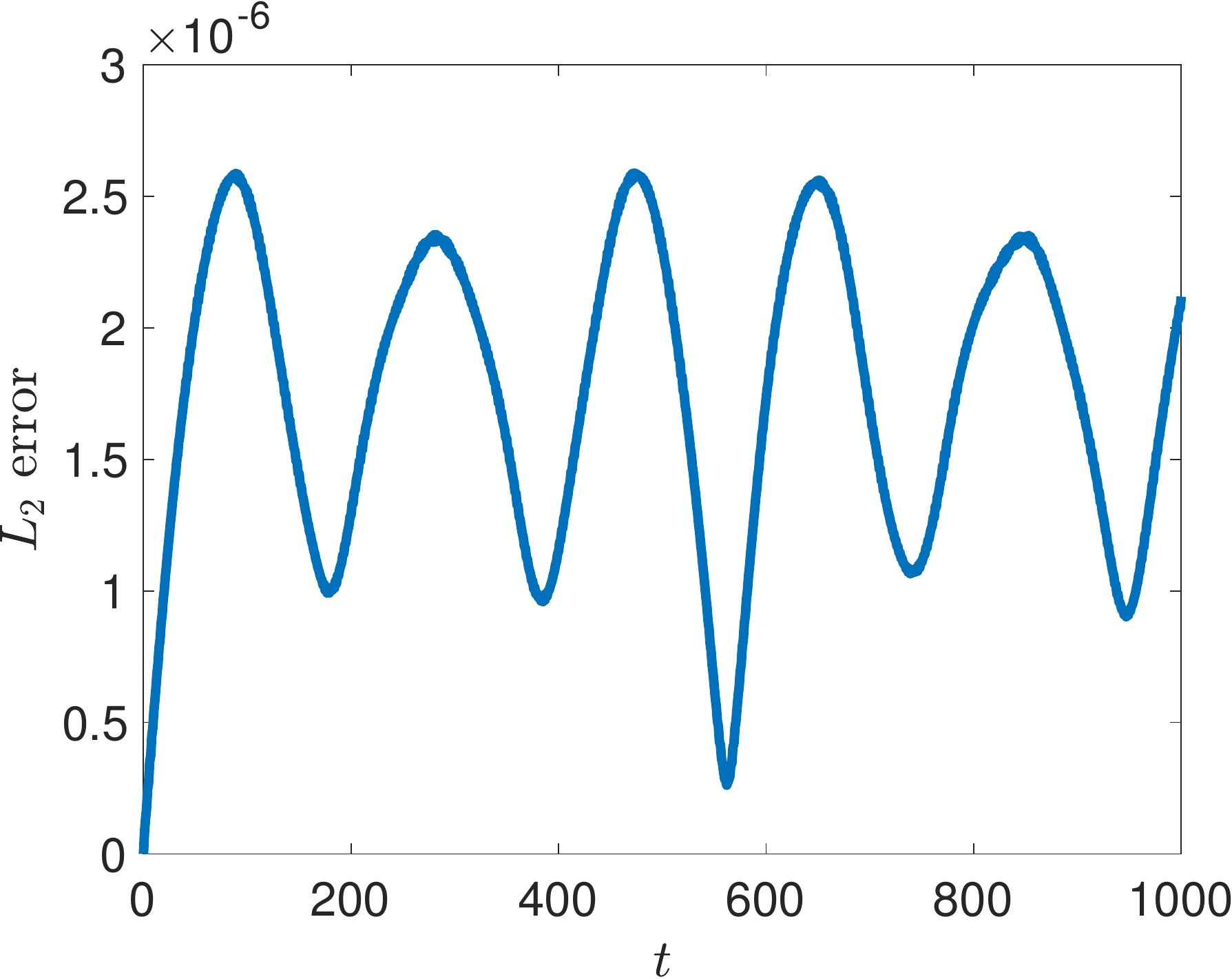}
\caption{$L_2$ error for the improved SBP-GP method for a long time simulation to time $t=1000$
  ($\sim$ 225 temporal periods).} % AP edited
\label{LongTimeSimulation}
\end{figure}

For the SBP-SAT method with three penalty terms, the stability limit appears to be $\delta_t\leq 1.18h$, 
% AP edited
which represents approximately a 45\% reduction in the time step. When using four penalty terms
and the sharper penalty parameters \cite{Almquist2019}, the scheme is stable for $\delta_t\leq
1.82h$, which is an improvement from the scheme with three penalty terms, but not as good as the
SBP-GP method.

%This indicates that the non-periodic boundary condition and the non-conforming grid interface do not affect time step restriction of the SBP-GP method, but the time step in the SBP-SAT method must be reduced significantly.

\subsubsection{Conditioning and sparsity of the linear system for ghost points}
In the SBP-GP method, a system of linear equations needs to be solved to compute the solution at the %AP
ghost points. To demonstrate the superiority of the improved SBP-GP method, we examine the
conditioning and sparsity of the system on three meshes.

\begin{table}
\centering
\begin{tabular}{lllll}
\toprule
$N_c$ & $cond_i$ & $cond_o$ & $nnz_i$ & $nnz_o$\\ \midrule
$321^2$ & 1.26 & 778 &2240 & 4160\\
$641^2$ & 1.26 & 1680 & 4480 & 8320\\
$1281^2$ & 1.26 & 3425& 8960 &16640\\
\bottomrule
\end{tabular}
\caption{Condition number $cond$ and number of nonzero elements $nnz$ in the matrix for ghost points. The subscript $o$ and $i$ correspond to the original and improved SBP-GP method, respectively. $N_c$ denotes the number of grid points in the coarse domain.}
\label{tab_cond}
\end{table}

In Table \ref{tab_cond}, we observe that for the improved SBP-GP method, the condition number is close to one and is independent of the mesh size. In contrast, the condition number in the original SBP-GP method is several magnitudes larger, and grows with mesh refinement. Furthermore, the number of nonzero elements in the improved SBP-GP matrix is approximately half the number of nonzero elements in the matrix in the original method. Hence, the system of linear equations in the improved SBP-GP method is both more sparse and better conditioned. 

\subsubsection{Convergence rate}
We now perform a convergence study for the SBP-GP method and the SBP-SAT method. We choose the time
step $\delta_t = h$ so that both methods are stable.  The $L^2$ errors in the numerical solution
with the SBP-GP method are shown in Table \ref{tab1}. Though the dominating truncation error is
$\mathcal{O}(h^2)$ at grid points near boundaries, the numerical solution converges to fourth order
accuracy, i.e. two orders are gained in convergence rate \cite{Wang2017a}. %AP
\begin{table}
\centering
\begin{tabular}{ll}
\toprule
$2h$ & $L^2$ error (rate)\\ 
\midrule
1.57$\times  10^{-1}$ &  1.6439$\times 10^{-3}$\\
7.85$\times  10^{-2}$ &  1.0076$\times 10^{-4}$ (4.02)\\
3.93$\times  10^{-2}$ &  6.2738$\times 10^{-6}$ (4.01)\\
1.96$\times  10^{-2}$ &  3.9193$\times 10^{-7}$ (4.00)\\
9.81$\times  10^{-3}$ &  2.4344$\times 10^{-8}$ (4.01)\\
\bottomrule
\end{tabular}
\caption{$L^2$ errors (convergence rates) of the fourth order SBP-GP method for piecewise constant $\mu$.}
\label{tab1}
\end{table}

For the SBP-SAT method with three penalty terms \eqref{SAT_INTf}-\eqref{SATc}, the $L^2$ errors labeled as SAT3 in Table \ref{tab_4th_sat} only converge at a rate of three. Because the dominating  truncation error is $\mathcal{O}(h^2)$ at grid points close to boundaries, we gain only one order of accuracy in the numerical solution. This suboptimal convergence behavior has also been observed in other settings \cite{Wang2017a}.

\begin{table}
\centering
\begin{tabular}{llll}
\toprule
$2h$ & $L^2$ error (rate) SAT3 & $L^2$ error (rate) SAT4 & $L^2$ error (rate) INT6 \\ 
\midrule
1.57$\times  10^{-1}$ & 3.0832$\times  10^{-3}$           & 2.1104$\times 10^{-3}$ & 2.1022$\times  10^{-3}$  \\
7.85$\times  10^{-2}$ & 3.4792$\times 10^{-4}$ (3.15) & 1.1042$\times 10^{-4}$ (4.26)  & 1.1014$\times 10^{-4}$ (4.25) \\
3.93$\times  10^{-2}$ &  4.4189$\times 10^{-5}$ (2.98) & 6.6902$\times 10^{-6}$ (4.04) & 6.6815$\times 10^{-6}$ (4.04)\\
1.96$\times  10^{-2}$ &  5.6079$\times 10^{-6}$ (2.98) & 4.0374$\times 10^{-7}$ (4.05) & 4.0346$\times 10^{-7}$ (4.05)\\
9.81$\times  10^{-3}$ & 7.0745$\times 10^{-7}$ (2.99) & 2.4659$\times 10^{-8}$ (4.03) & 2.4651$\times 10^{-8}$ (4.03) \\
\bottomrule
\end{tabular}
\caption{$L^2$ errors (convergence rates) of the fourth order SBP-SAT method  for piecewise constant $\mu$.}
\label{tab_4th_sat}
\end{table}

% AP made some edits
We have found two simple remedies to obtain a fourth order convergence rate. First, when using the
SBP-SAT method with four penalty terms, we obtain a fourth order convergence rate, as shown in the third
column of Table \ref{tab_4th_sat} labeled as SAT4. Alternatively, we can use three penalty terms but
employ a sixth order interpolation and restriction operators at the non-conforming interface. This also leads
to a fourth order convergence rate, see the fourth column of Table \ref{tab_4th_sat}, labeled
INT6. In both approaches, the dominating truncation error is still $\mathcal{O}(h^2)$ at a few grid
points close to the boundaries. However, different penalty terms will give different boundary
systems in the normal mode analysis for convergence rate. The precise rate of convergence can be
analyzed by the Laplace-transform method, but is beyond the scope of this paper.

We also observe that the $L^2$ errors of the SBP-GP method is almost identical to that of the
SBP-SAT method (SAT4 and INT6) with the same mesh size.

\subsection{Smooth material parameters}\label{ex_smooth}
In this section, we test the two methods when the material parameters are smooth functions in the
whole domain $\Omega$. More precisely, we use material parameters
\begin{equation*}
\begin{split}
\rho&=-\cos(x)\cos(y)+3, \\
\mu&=\cos(x)\cos(y)+2.
\end{split}
\end{equation*}
 The forcing function and initial conditions are chosen so that the manufactured solution becomes % AP
\begin{equation*}
u(x,y,t) = \sin(x+2)\cos(y+1)\sin(t+3).
\end{equation*}
We use the same grid as in Section \ref{ex_Snell} with grid size $2h$ in $\Omega^1$ and $h$ in
$\Omega^2$. The parameters $\rho_{\min} = 2$ and $\mu_{\max}=3$ take the extreme values at the same
grid point. Therefore, a Fourier analysis of the corresponding periodic problem gives the time step %AP
restriction
\begin{equation*}
\delta_t\leq \frac{1}{\sqrt{2}}\frac{2\sqrt{3}}{\sqrt{16/(3h^2)}\sqrt{\mu_{\max}/\rho_{\min}}}=\frac{\sqrt{3}}{2}h\approx 0.86h.
\end{equation*}
Numerically, we have found that the SBP-GP method is stable when $\delta_t\leq 0.86h$. This shows again that the non--periodicity and interface coupling do not affect the CFL condition in the SBP-GP method.  The SBP-SAT method is stable with $\delta_t\leq 0.77h$, which means that the time step needs to be reduced by approximately 10\%. 

To test convergence, we choose the time step $\delta_t= 0.7h$ so that both the SBP-GP method and SBP-SAT method are stable. The $L^2$ errors at $t=11$ are shown in Table \ref{tab3} for the SBP-GP method. We observe a fourth order convergence rate.
\begin{table}
\centering
\begin{tabular}{ll}
\toprule
$2h$ & $L^2$ error (rate)\\ 
\midrule
1.57$\times  10^{-1}$ & 2.7076$\times 10^{-4}$  \\
7.85$\times  10^{-2}$ & 1.6000$\times 10^{-5}$ (4.08) \\
3.93$\times  10^{-2}$ & 9.7412$\times 10^{-7}$ (4.04) \\
1.96$\times  10^{-2}$ & 6.0183$\times 10^{-8}$ (4.02) \\
9.81$\times  10^{-3}$ & 3.7426$\times 10^{-9}$  (4.01) \\
\bottomrule
\end{tabular}
\caption{$L^2$ errors (convergence rates) of the SBP-GP method for smooth $\mu$.}
\label{tab3}
\end{table}

Similar to the case with piecewise constant material property, the standard SBP-SAT method only
converges to third order accuracy, see the second column of Table \ref{tab4} labeled as SAT3. We
have tested the SBP-SAT method with four penalty terms, or with a sixth order interpolation and
restriction operator. Both methods lead to a fourth order convergence rate, see the third and fourth
column in Table \ref{tab4}.  However, the $L^2$ error is more than three times as large as the $L^2$ error
of the SBP-GP method with the same mesh size. %AP
\begin{table}
\centering
\begin{tabular}{llll}
\toprule
$2h$ & $L^2$ error (rate) SAT3 & $L^2$ error (rate) SAT4 & $L^2$ error (rate) INT6 \\ 
\midrule
1.57$\times  10^{-1}$ & 3.8636$\times  10^{-3}$           & 1.8502$\times 10^{-3}$ & 1.8503$\times  10^{-3}$  \\
7.85$\times  10^{-2}$ & 4.3496$\times 10^{-4}$ (3.15) & 9.4729$\times 10^{-5}$ (4.29)  & 9.4736$\times 10^{-5}$ (4.29) \\
3.93$\times  10^{-2}$ &  5.3152$\times 10^{-5}$ (3.03) & 3.7040$\times 10^{-6}$ (4.68) & 3.7043$\times 10^{-6}$ (4.68)\\
1.96$\times  10^{-2}$ &  6.6271$\times 10^{-6}$ (3.00) & 2.0778$\times 10^{-7}$ (4.16) & 2.0779$\times 10^{-7}$ (4.16)\\
9.81$\times  10^{-3}$ & 8.2783$\times 10^{-7}$ (3.00) & 1.3372$\times 10^{-8}$ (3.96) & 1.3372$\times 10^{-8}$ (3.96) \\
\bottomrule
\end{tabular}
\caption{$L^2$ errors (convergence rates) of the fourth order SBP-SAT method for smooth $\mu$.}
\label{tab4}
\end{table}

\section{Conclusion}\label{sec_conc}

We have analyzed two different types of SBP finite difference operators for solving the wave
equation with variable coefficients: operators with ghost points, $\widetilde G(\mu)$, and operators without
ghost points, $G(\mu)$. The close relation between the two operators has been analyzed
and we have presented a way of adding or removing the ghost point dependence in the
operators. Traditionally, the two operators have been used within different approaches for imposing
the boundary conditions. Based on their relation, we have in this paper devised a scheme that
combines both operators for satisfying the interface conditions at a non-conforming grid refinement
interface.

We first used the SBP operator with ghost points to derive a fourth order accurate SBP-GP
method for the wave equation with a grid refinement interface. This method uses ghost points from
both sides of the refinement interface to enforce the interface conditions. Accuracy and stability
of the method are ensured by using a fourth order accurate interpolation stencil and a compatible
restriction stencil. Secondly, we presented an improved method, where only ghost points from the
coarse side are used to impose the interface conditions. This is achieved by combining the operator
$G(\mu)$ in the fine grid and the operator $\widetilde G(\mu)$ in the coarse grid. Compared to the % AP: reversed grids
first SBP-GP method, the improved method leads to a smaller system of linear equations for the ghost
points with better conditioning. In addition, we have made improvements to the traditional fourth order SBP-SAT method, which
only exhibits a third order convergence rate for the wave equation with a grid refinement
interface. Two remedies have been presented and both result in a fourth order convergence rate.

We have conducted numerical experiments to verify that the proposed methods converge with fourth
order accuracy, for both smooth and discontinuous material properties. With a discontinuous
material, the domain is partitioned into subdomains such that discontinuities are aligned with
subdomain boundaries. We have also found numerically that the proposed SBP-GP method is stable under
a CFL time-step condition that is very close to the von Neumann limit for the corresponding periodic
problem. Being able to use a large time step is essential for solving practical large-scale wave
propagation problems, because the computational complexity grows linearly with the number of time
steps. We have found that the SBP-SAT method requires a smaller time step for stability, and that the
time step depends on the penalty parameters of the interface coupling conditions. In the case of
smooth material properties, the SBP-SAT method was also found to yield to a larger solution %AP
error compared to the SBP-GP method, for the same grid sizes and time step.

One disadvantage of the SBP-GP method is that a system of linear equations must be solved to obtain
the numerical solutions at the ghost points. However, previous work has demonstrated that the system
can be solved very efficiently by an iterative method
\cite{Petersson2015,Petersson2018}. Furthermore, the proposed method only uses ghost points on one
side of the interface and therefore leads to a linear system with fewer unknowns and a more regular
structure than previously. 

%AP: sixth order!
Sixth order accurate SBP operators can be used in the proposed method in a straightforward
way. However, sixth order SBP discretization often leads to a convergence rate lower than six, and it
is an open question if a six order discretization is more efficient than a fourth order
discretization for realistic problems. In future work we plan to extend the proposed method to the
elastic wave equation in three space dimensions with realistic topography based on
\cite{Petersson2015}, and implement it on a distributed memory machine to evaluate its efficiency. %AP

%One disadvantage of the SBP-GP method is that a system of linear equations must be solved to obtain
%the numerical solutions at the ghost points. However, previous work has demonstrated that the system
%can be solved very efficiently by an iterative method
%\cite{Petersson2015,Petersson2018}. Furthermore, the proposed method only uses ghost points on one
%side of the interface and therefore leads to a linear system with fewer unknowns and a more regular
%structure than previously. In future work we plan to implement the proposed method for the elastic
%wave equation in three space dimensions on a distributed memory machine and evaluate its efficiency.

%Our proposed method leads to a
%system whose size is only 1/3 of the existing method, which will likely improve efficiency of the
%SBP-GP method.

\section*{Acknowledgments}
S.~Wang would like to thank Professor Gunilla Kreiss at Uppsala University for the support
of this project. Part of the work was conducted when S.~Wang was on a research visit at Lawrence
Livermore National Laboratory. The authors thank B.~Sj\"ogreen for sharing his unpublished work on the
SBP-GP method with ghost points on both sides of the grid refinement interface. This work was
performed under the auspices of the U.S.~Department of Energy by Lawrence Livermore National
Laboratory under contract DE-AC52-07NA27344. This is contribution LLNL-JRNL-757334.

\section*{Appendix 1: Proof of Lemma \ref{lemmaQ}}
By using the standard fourth order finite difference stencil, \eqref{1dwave_ne} can be approximated as 
\[
\frac{d^2 u_j}{dt^2}= \left(-\frac{1}{12}u_{j+2}+\frac{4}{3}u_{j+1}-\frac{5}{2}u_j+\frac{4}{3}u_{j-1}-\frac{1}{12}u_{j-2}\right)\frac{\mu}{\rho}.
\]
By using the ansatz $u_j=\hat u e^{i\omega x_j}$, where $\omega$ is the wave number and $x_j=jh$, we obtain 
%\[
%\frac{d^2\hat u}{dt^2} e^{i\omega x_{j}}=  \left(-\frac{1}{12}e^{i\omega x_{j+2}}+\frac{4}{3}e^{i\omega x_{j+1}}-\frac{5}{2}e^{i\omega x_{j}}+\frac{4}{3}e^{i\omega x_{j-1}}-\frac{1}{12}e^{i\omega x_{j-2}}\right)\hat u.
%\]
\begin{align*}
\frac{d^2\hat u}{dt^2} &=  \left(-\frac{1}{12}e^{i\omega 2h}+\frac{4}{3}e^{i\omega h}-\frac{5}{2}+\frac{4}{3}e^{-i\omega h}-\frac{1}{12}e^{-i\omega 2h}\right)\frac{\mu}{\rho}\hat u\\
&= -\frac{4}{h^2}\sin^2\frac{\omega h}{2}\left(1+\frac{1}{3}\sin^2\frac{\omega h}{2}\right)\frac{\mu}{\rho}\hat u.
\end{align*}
Therefore, the Fourier transform of the fourth order accurate central finite difference stencil is 
\begin{equation}\label{symbol}
\hat Q = -\frac{4}{h^2}\sin^2\frac{\omega h}{2}\left(1+\frac{1}{3}\sin^2\frac{\omega h}{2}\right)\frac{\mu}{\rho}.
\end{equation}
Consequently, we have 
\[
\kappa=\max |\hat Q|=\frac{16\mu}{3h^2\rho}.
\]
\bibliography{Siyang_References}

\begin{thebibliography}{10}

\bibitem{Almquist2019}
{\sc M.~Almquist, S.~Wang, and J.~Werpers}, {\em Order-preserving interpolation
  for summation-by-parts operators at nonconforming grid interfaces}, SIAM J.
  Sci. Comput., 41 (2019), pp.~A1201--A1227.

\bibitem{Appelo2007}
{\sc D.~Appel\"{o} and G.~Kreiss}, {\em Application of a perfectly matched
  layer to the nonlinear wave equation}, Wave Motion, 44 (2007), pp.~531--548.

\bibitem{Carpenter1994}
{\sc M.~H. Carpenter, D.~Gottlieb, and S.~Abarbanel}, {\em Time--stable
  boundary conditions for finite--difference schemes solving hyperbolic
  systems: methodology and application to high--order compact schemes}, J.
  Comput. Phys., 111 (1994), pp.~220--236.

\bibitem{Duru2014}
{\sc K.~Duru, G.~Kreiss, and K.~Mattsson}, {\em Stable and high--order accurate
  boundary treatments for the elastic wave equation on second--order form},
  SIAM J. Sci. Comput., 36 (2014), pp.~A2787--A2818.

\bibitem{Duru2014V}
{\sc K.~Duru and K.~Virta}, {\em Stable and high order accurate difference
  methods for the elastic wave equation in discontinuous media}, J. Comput.
  Phys., 279 (2014), pp.~37--62.

\bibitem{Gilbert2008}
{\sc J.~C. Gilbert and P.~Joly}, {\em Higher order time stepping for second
  order hyperbolic problems and optimal {CFL} conditions}, Springer, 2008,
  pp.~67--93.

\bibitem{Graff1991}
{\sc K.~F. Graff}, {\em Wave Motion in Elastic Solids}, Dover Publications,
  1991.

\bibitem{Gustafsson1975}
{\sc B.~Gustafsson}, {\em The convergence rate for difference approximations to
  mixed initial boundary value problems}, Math. Comput., 29 (1975),
  pp.~396--406.

\bibitem{Hagstrom2012}
{\sc T.~Hagstrom and G.~Hagstrom}, {\em Grid stabilization of high--order
  one--sided differencing {II}: second--order wave equations}, J. Comput.
  Phys., 231 (2012), pp.~7907--7931.

\bibitem{Hesthaven2008}
{\sc J.~S. Hesthaven and T.~Warburton}, {\em Nodal Discontinuous Galerkin
  Methods}, Springer, 2008.

\bibitem{Kozdon2016}
{\sc J.~E. Kozdon and L.~C. Wilcox}, {\em Stable coupling of nonconforming,
  high--order finite difference methods}, SIAM J. Sci. Comput., 38 (2016),
  pp.~A923--A952.

\bibitem{Kreiss1972}
{\sc H.~O. Kreiss and J.~Oliger}, {\em {C}omparison of accurate methods for the
  integration of hyperbolic equations}, Tellus, 24 (1972), pp.~199--215.

\bibitem{Kreiss2002}
{\sc H.~O. Kreiss, N.~A. Petersson, and J.~Ystr\"{o}m}, {\em Difference
  approximations for the second order wave equation}, SIAM. J. Numer. Anal., 40
  (2002), pp.~1940--1967.

\bibitem{Kreiss1974}
{\sc H.~O. Kreiss and G.~Scherer}, {\em Finite element and finite difference
  methods for hyperbolic partial differential equations}, Mathematical Aspects
  of Finite Elements in Partial Differential Equations, Symposium Proceedings,
  (1974), pp.~195--212.

\bibitem{Mattsson2012}
{\sc K.~Mattsson}, {\em Summation by parts operators for finite difference
  approximations of second--derivatives with variable coefficient}, J. Sci.
  Comput., 51 (2012), pp.~650--682.

\bibitem{Mattsson2013}
{\sc K.~Mattsson and M.~Almquist}, {\em A solution to the stability issues with
  block norm summation by parts operators}, J. Comput. Phys., 253 (2013),
  pp.~418--442.

\bibitem{Mattsson2008}
{\sc K.~Mattsson, F.~Ham, and G.~Iaccarino}, {\em Stable and accurate
  wave--propagation in discontinuous media}, J. Comput. Phys., 227 (2008),
  pp.~8753--8767.

\bibitem{Mattsson2009}
{\sc K.~Mattsson, F.~Ham, and G.~Iaccarino}, {\em Stable boundary treatment for
  the wave equation on second--order form}, J. Sci. Comput., 41 (2009),
  pp.~366--383.

\bibitem{Mattsson2004}
{\sc K.~Mattsson and J.~Nordstr\"{o}m}, {\em Summation by parts operators for
  finite difference approximations of second derivatives}, J. Comput. Phys.,
  199 (2004), pp.~503--540.

\bibitem{Olsson1995a}
{\sc P.~Olsson}, {\em Summation by parts, projections, and stability. {I}},
  Math. Comput., 64 (1995), pp.~1035--1065.

\bibitem{Olsson1995b}
{\sc P.~Olsson}, {\em Summation by parts, projections, and stability. {II}},
  Math. Comput., 64 (1995), pp.~1473--1493.

\bibitem{Petersson2010}
{\sc N.~A. Petersson and B.~Sj\"{o}green}, {\em Stable grid refinement and
  singular source discretization for seismic wave simulations}, Commun. Comput.
  Phys., 8 (2010), pp.~1074--1110.

\bibitem{Petersson2015}
{\sc N.~A. Petersson and B.~Sj\"{o}green}, {\em Wave propagation in anisotropic
  elastic materials and curvilinear coordinates using a summation--by--parts
  finite difference method}, J. Comput. Phys., 299 (2015), pp.~820--841.

\bibitem{Sw4}
{\sc N.~A. Petersson and B.~Sj\"ogreen}, {\em User's guide to {SW4}, version
  2.0}, Tech. Report LLNL-SM-741439, Lawrence Livermore National Laboratory,
  2016.
\newblock (Source code available from {\tt geodynamics.org/cig}).

\bibitem{Petersson2018}
{\sc N.~A. Petersson and B.~Sj\"{o}green}, {\em High order accurate finite
  difference modeling of seismo-acoustic wave propagation in a moving
  atmosphere and a heterogeneous earth model coupled across a realistic
  topography}, J. Sci. Comput., 74 (2018), pp.~290--323.

\bibitem{Sjogreen2012}
{\sc B.~Sj\"{o}green and N.~A. Petersson}, {\em A fourth order accurate finite
  difference scheme for the elastic wave equation in second order formulation},
  J. Sci. Comput., 52 (2012), pp.~17--48.

\bibitem{Virta2014}
{\sc K.~Virta and K.~Mattsson}, {\em Acoustic wave propagation in complicated
  geometries and heterogeneous media}, J. Sci. Comput., 61 (2014), pp.~90--118.

\bibitem{Wang2018}
{\sc S.~Wang}, {\em An improved high order finite difference method for
  non--conforming grid interfaces for the wave equation}, J. Sci. Comput., 77
  (2018), pp.~775--792.

\bibitem{Wang2017a}
{\sc S.~Wang and G.~Kreiss}, {\em Convergence of summation--by--parts finite
  difference methods for the wave equation}, J. Sci. Comput., 71 (2017),
  pp.~219--245.

\bibitem{Wang2017b}
{\sc S.~Wang, A.~Nissen, and G.~Kreiss}, {\em Convergence of finite difference
  methods for the wave equation in two space dimensions}, Math. Comp., 87
  (2018), pp.~2737--2763.

\bibitem{Wang2016}
{\sc S.~Wang, K.~Virta, and G.~Kreiss}, {\em {H}igh order finite difference
  methods for the wave equation with non--conforming grid interfaces}, J. Sci.
  Comput., 68 (2016), pp.~1002--1028.

\end{thebibliography}
\bibliographystyle{siamplain}

\end{document}